\documentclass[12pt]{article}

\usepackage{times}

\usepackage[shortlabels]{enumitem}
\usepackage{enumitem}
\usepackage[utf8]{inputenc}
\usepackage{commath}
\usepackage{amsthm, amscd}
\usepackage{amsmath}
\usepackage{amssymb}
\usepackage{cite}
\usepackage{setspace}
\usepackage{ stmaryrd }
\usepackage{tikz-cd}

\usepackage{tocloft}
\usepackage{sectsty}
\usepackage{xparse}

\newcommand*\tasklabelformat[1]{#1)}

\usepackage{titlesec}

\usepackage{tasks}[newest]
\settasks{
  label = \alph* ,
  label-format = \tasklabelformat ,
  label-width  = 12pt
}

\usepackage{bigints}
\usepackage{comment}
\usepackage{mathtools}
\usepackage{mathrsfs}
\usepackage{fancyhdr}

\allowdisplaybreaks[4]

\numberwithin{equation}{section}
\usepackage{etoolbox}
\patchcmd{\thebibliography}
  {\settowidth}
  {\setlength{\itemsep}{0pt plus -10pt}\settowidth}
  {}{}
\apptocmd{\thebibliography}
  {
  }
  {}{}
\makeatletter
\newtheorem*{rep@theorem}{\rep@title}
\newcommand{\newreptheorem}[2]{%
\newenvironment{rep#1}[1]{%
 \def\rep@title{#2 \ref{##1}}%
 \begin{rep@theorem}}%
 {\end{rep@theorem}}}
\makeatother

\theoremstyle{theorem}

\newreptheorem{theorem}{Theorem}
\newtheorem{thm}{Theorem}[section]
\newtheorem*{thm*}{Theorem}
\theoremstyle{definition}
\newtheorem{prop}[thm]{Proposition}
\newtheorem*{prop*}{Proposition}

\newtheorem{lem}[thm]{Lemma}
\newtheorem{cor}[thm]{Corollary}
\newtheorem*{cor*}{Corollary}
\theoremstyle{remark}

\newtheorem{rem}[thm]{Remark}

\title{\vspace*{-1.5cm} Kodaira-Iitaka dimension and multiplicity: \\
an analytic perspective}

\author
{Siarhei Finski
}

\date{}

\usepackage[%
    left=1in,%
    right=1in,%
    top=1.1in,%
    bottom=0.8in,%
    paperheight=11in,%
    paperwidth=8.5in%
]{geometry}

\newcommand{\imun} {\sqrt{-1}}
\newcommand{\vol}{v}

\newcommand{\sym}{{\rm{Sym}}}

\newcommand{\comp}{\mathbb{C}}
\newcommand{\real}{\mathbb{R}}

\newcommand{\nat}{\mathbb{N}}

\newcommand{\dbar}{ \overline{\partial} }

\newcommand{\ddc}{\mathrm{d} \mathrm{d}^c}

\renewcommand{\Im}{\operatorname{Im}}

\newcommand{\ban}{\operatorname{Ban}}

\newcommand{\psh}{\operatorname{PSH}}

\makeatletter
\DeclareFontFamily{OMX}{MnSymbolE}{}
\DeclareSymbolFont{MnLargeSymbols}{OMX}{MnSymbolE}{m}{n}
\SetSymbolFont{MnLargeSymbols}{bold}{OMX}{MnSymbolE}{b}{n}
\DeclareFontShape{OMX}{MnSymbolE}{m}{n}{
    <-6>  MnSymbolE5
   <6-7>  MnSymbolE6
   <7-8>  MnSymbolE7
   <8-9>  MnSymbolE8
   <9-10> MnSymbolE9
  <10-12> MnSymbolE10
  <12->   MnSymbolE12
}{}
\DeclareFontShape{OMX}{MnSymbolE}{b}{n}{
    <-6>  MnSymbolE-Bold5
   <6-7>  MnSymbolE-Bold6
   <7-8>  MnSymbolE-Bold7
   <8-9>  MnSymbolE-Bold8
   <9-10> MnSymbolE-Bold9
  <10-12> MnSymbolE-Bold10
  <12->   MnSymbolE-Bold12
}{}

\let\llangle\@undefined
\let\rrangle\@undefined
\DeclareMathDelimiter{\llangle}{\mathopen}%
                     {MnLargeSymbols}{'164}{MnLargeSymbols}{'164}
\DeclareMathDelimiter{\rrangle}{\mathclose}%
                     {MnLargeSymbols}{'171}{MnLargeSymbols}{'171}
\makeatother

\newenvironment{sciabstract}{}

\cftsetindents{section}{0em}{1.7em}

\sectionfont{\large}

\titlespacing*{\section}
{0pt}{5pt}{5pt}

\begin{document}

\maketitle 

\vspace*{-0.7cm}

\vspace*{0.3cm}

\begin{sciabstract}
  \textbf{Abstract.} 
  	We express the Kodaira-Iitaka dimension and the multiplicity of graded linear series in terms of the intersection theory of the plurisubharmonic envelope associated with the linear series, and obtain two refined versions of these formulas at the pointwise and at the metric levels.
	\par 
	At the pointwise level, we focus on the weak convergence of the partial Bergman kernel associated with the linear series and a Bernstein-Markov measure.
	At the metric level, we compute the asymptotic ratio of the volumes of unit balls defined by the sup-norms on the linear series.
	\par 
	Based on our findings, we introduce a non-pluripolar version of the numerical Kodaira-Iitaka dimension for a line bundle, show that this invariant dominates the classical Kodaira-Iitaka dimension and is, in turn, bounded above by the numerical versions proposed so far.
\end{sciabstract}

\pagestyle{fancy}
\lhead{}
\chead{Kodaira-Iitaka dimension and multiplicity: an analytic perspective}
\rhead{\thepage}
\cfoot{}


\newcommand{\Addresses}{{
  \bigskip
  \footnotesize
  \noindent \textsc{Siarhei Finski, CNRS-CMLS, École Polytechnique F-91128 Palaiseau Cedex, France.}\par\nopagebreak
  \noindent  \textit{E-mail }: \texttt{finski.siarhei@gmail.com}.
}} 

\vspace*{0.25cm}

\par\noindent\rule{1.25em}{0.4pt} \textbf{Table of contents} \hrulefill

\vspace*{-1.5cm}

\tableofcontents

\vspace*{-0.2cm}

\noindent \hrulefill


\section{Introduction}\label{sect_intro}
	We fix a connected complex projective manifold $X$, $\dim_{\comp} X = n$, and a holomorphic line bundle $L$ over $X$.
	The main object of the present study is a \textit{(graded) linear series} $W = \oplus_{k = 0}^{\infty} W_k$, $W_0 = \comp$, defined as a subring of the associated \textit{section ring}
	\begin{equation}
		R(X, L) := \oplus_{k = 0}^{\infty} H^0(X, L^{\otimes k}).
	\end{equation}
	To exclude some trivial cases, we make an additional assumption that $W \neq \comp$, implying in particular that $L$ is \textit{effective}, i.e., $H^0(X, L^{\otimes k}) \neq \{0\}$ for $k \in \nat^*$ large enough.
	Without loosing the generality (as otherwise one could restrict to a tensor power of $L$), we further assume
	\begin{equation}\label{eq_lin_series_ndg}
		W_k \neq \{ 0 \}, \quad \text{for } k \in \nat \text{ large enough}.
	\end{equation}
	\par 
	Using Okounkov bodies, Kaveh-Khovanskii in \cite[Theorem 3.4]{KavehKhov} established that there is $\kappa(W) \in \{0, 1, \ldots, n\}$, called the \textit{Kodaira-Iitaka dimension} of $W$, so that the following limit
	\begin{equation}\label{eq_exist_volume}
		{\rm{vol}}_{\kappa}(W)
		:=
		\lim_{k \to \infty} \frac{\dim W_k}{k^{\kappa(W)} / \kappa(W)!}
	\end{equation}
	exists and is positive; see also Lazarsfeld-Mustață \cite{LazMus} for the regime $\kappa(W) = n$, and Iitaka \cite{Iitaka}, Fujita \cite{Fujita} for the previous works in realms of the complete linear series, that is, $W = R(X,L)$.
	We call the limit in (\ref{eq_exist_volume}) the \textit{multiplicity} of $W$.
	When $W$ is a complete linear series, we call the above quantity the multiplicity of $L$ and denote it by ${\rm{vol}}_{\kappa}(L)$. 
	The corresponding Kodaira-Iitaka dimension is denoted by $\kappa(L)$ and is called the Kodaira-Iitaka dimension of $L$.
	\par 
	This paper has three complementary aims: to establish analytic formulas for the Kodaira-Iitaka dimension and the multiplicity of graded linear series, and to refine these formulas at both the pointwise and metric levels.
	By the pointwise level, we refer to the asymptotic analysis of the corresponding partial Bergman kernels. 
	By the metric level, we mean comparing the sequences of norms on $W_k$ induced by the sup-norms associated with two continuous metrics on $L$.
	\par 
	We begin with the former formula, which will be based on a notion of a plurisubharmonic (psh) envelope of a metric associated with the linear series.
	The latter definition can be traced back to Siciak \cite{SiciakEnvFirst}; in the current context of linear series it was introduced by Hisamoto \cite{HisamVolume}.
	To recall it, for any norm $N_k$ on $W_k$, we associate a singular metric $FS(N_k)$ on $L^{\otimes k}$ as
	\begin{equation}\label{eq_fs_norm}
		|l|_{FS(N_k), x}
		=
		\inf_{\substack{s \in H^0(X, L^{\otimes k}) \\ s(x) = l}}
		\| s \|_{N_k}, 
		\quad 
		\text{for any } x \in X, l \in L^{\otimes k}_x,
	\end{equation}
	with the convention that the infimum equals $+ \infty$ if there is no $s \in W_k$ such that $s(x) = l$.
	\par 
	It is easy to see, see Lemma \ref{lem_fs_sm}, that for any continuous metric $h^L$ on $L$, the sequence of Fubini-Study (singular) metrics associated with the sup-norms $\textrm{Ban}_k^{\infty}[W](h^L)$ on $W_k$, $k \in \nat^*$, is submultiplicative.
	In particular, by the Fekete's lemma, the following limit
	\begin{equation}\label{eq_env_w}
		P[W, h^L]
		=
		\Big(
		\lim_{k \to \infty} FS({\textrm{Ban}}_k^{\infty}[W](h^L))^{\frac{1}{k}}
		\Big)_*
	\end{equation}
	exists as a (singular) metric on $L$, where $(\cdot)_*$ is the lower semicontinuous regularization.
	Standard results imply that $P[W, h^L]$ has a psh potential, \cite[Proposition I.4.24]{DemCompl}.
	\par
	Define the closed positive $(1,1)$-current $c_1(L, P[W, h^L])$ as described after (\ref{eq_c1_bc}).
	On $X$, fix a Kähler form $\omega$ and consider the following measures, defined in the non-pluripolar, see \cite{BEGZ}, sense
	\begin{equation}\label{eq_meas_x_wedge}
		c_1(L, P[W, h^L])^{i} \wedge \omega^{n - i}, \qquad i \in \{0, 1, \ldots, n\}.
	\end{equation}
	Although the above measures in general depend on $\omega$, many of the associated quantities turn out to be independent of it, as shown below.
	\begin{thm}\label{thm_ki_form}
		For any graded linear series $W \subset R(X, L)$, we have
		\begin{equation}\label{eq_kappa_w}	
			\kappa(W) = \max\Big\{
				i \in \{ 0, 1, \ldots, n \} :  c_1(L, P[W, h^L])^{i} \wedge \omega^{n - i} \neq 0
			\Big\}.
		\end{equation}
	\end{thm}
	\begin{rem}
		At first sight, the condition in (\ref{eq_kappa_w}) appears to depend on the choice of the Kähler form $\omega$ and of the metric $h^L$, but this is not the case; see the end of Section \ref{sect_fs_inter} for details.
	\end{rem}
	\par 
	Our next result gives an analytic expression for the multiplicity of a graded linear series.
	It will involve another important invariant of the linear series, which we now recall.
	\par 
	Recall that any linear series defines the rational maps $X \dashrightarrow \mathbb{P}(W_m^*)$ that send a point $x \in X$ to the hyperplane $H_{m, x} \subset W_m$ consisting of all sections in $W_m$ that vanish at $x$ (and are undefined whenever all sections in $W_m$ vanish at $x$).
	We denote by $Y_m$ the closure of the image of this map.
	A theorem of Chang-Jow \cite[Theorem 1.1]{ChangJow} states that the rational maps 
	\begin{equation}\label{eq_ki_maps}
		\varphi_m : X \dashrightarrow Y_m
	\end{equation}
	are birationally equivalent for $m$ large enough.
	It is well-known, cf. \cite[Theorem 3.15]{BoucksomBourbaki}, that 
	\begin{equation}\label{eq_dim_ki_yk}
		\kappa(W) = \dim(Y_m), \quad \text{for } m \in \nat \text{ large enough}.
	\end{equation}
	We say that $W$ is \textit{birational} if the maps (\ref{eq_ki_maps}) are birational for $m$ large enough, cf. \cite[Definition 2.5]{LazMus}.
	Taking into account (\ref{eq_dim_ki_yk}), we see that birational linear series only exist when $L$ is big, i.e., such that the Kodaira-Iitaka dimension of $L$ is $n$.
	\par 
	We denote by $[\omega] \in H^{1, 1}(X) \cap H^2(X, \real)$ the cohomology class of $\omega$, and by $\varphi_{*} [\omega]^{n - \kappa(W)}$  the fiber integral of $\omega$ along a general fiber of $\varphi_m$ for $m$ large enough.
	By the result recalled after (\ref{eq_ki_maps}), this quantity depends only on the cohomology class $[\omega]$, as indicated by the notation.
	\begin{thm}\label{thm_vol_form}
		For any graded linear series $W \subset R(X, L)$, we have
		\begin{equation}\label{eq_vol_form}
			{\rm{vol}}_{\kappa}(W)
			=
			\frac{1}{\varphi_{*} [\omega]^{n - \kappa(W)}} \int_X c_1(L, P[W, h^L])^{\kappa(W)} \wedge \omega^{n - \kappa(W)}.
		\end{equation}
	\end{thm}
	\begin{rem}
		a)
		When $W = R(X, L)$ and $L$ is big, by \cite{SiciakExtremal} and \cite[Theorem 2.16]{FinGQBig}, we have
		\begin{equation}\label{eq_env_ident}
			P[W, h^L] = P[h^L],
		\end{equation}
		where $P[h^L]$ is the psh envelope defined by
		\begin{equation}\label{eq_sic_env}
			P[h^L]
			=
			\inf \Big\{ h^L_0 : h^L_0 \text{ has a psh potential and } h^L_0 \geq h^L
			\Big\}.
		\end{equation}
		Note that immediately from the definition, $P[h^L]$ has minimal singularities in the sense described in Section \ref{sect_min_sing}.
		Theorem \ref{thm_vol_form}, Lemma \ref{lem_env_sing_vers}, (\ref{eq_env_ident}) and the monotonicity formula, cf. Theorem \ref{thm_wn_monot}, then recover a theorem of Boucksom \cite{BouckVol} which states that for a big line bundle $L$ and any positive closed $(1, 1)$-current $T_{\min}$ with a potential of minimal singularities in the class $[c_1(L)]$, we have
		\begin{equation}\label{eq_vol_bouck}
			{\rm{vol}}_{\kappa}(L) = \int_X T_{\min}^n.
		\end{equation}
		\par 
		b) When $W$ is birational, (\ref{eq_vol_form}) is due to Hisamoto \cite{HisamVolume}.
	\end{rem}
	\par 
	One may ask whether an analogue of (\ref{eq_vol_bouck}) remains valid for line bundles that are not big.	
	To discuss this in details, we introduce the non-pluripolar numerical Kodaira-Iitaka dimension, $\kappa_{{\rm{np}}}(L)$, and the non-pluripolar numerical multiplicity, ${\rm{vol}}_{\kappa, {\rm{np}}}(L)$, as follows
	\begin{equation}\label{eq_numerical_defn}
	\begin{aligned}
		&
		\kappa_{{\rm{np}}}(L) = \max\Big\{
			i \in \{ 0, 1, \ldots, n \} :  T_{\min}^{i} \wedge \omega^{n - i} \neq 0
		\Big\}.
		\\
		&
		{\rm{vol}}_{\kappa, {\rm{np}}}(L)
		=
		\frac{1}{\varphi_{*} [\omega]^{n - \kappa(L)}} \int_X T_{\min}^{\kappa(L)} \wedge \omega^{n - \kappa(L)},
	\end{aligned}
	\end{equation}
	where $T_{\min}$ is a positive closed $(1, 1)$-current with a potential of minimal singularities in the class $[c_1(L)]$ and $\varphi_{*} [\omega]^{n - \kappa(L)}$ is defined for the complete linear series as before.
	Note that both of these notions are independent of the choice of $T_{\min}$ by the monotonicity formula, cf. Theorem \ref{thm_wn_monot}, which along with Theorems \ref{thm_ki_form} and \ref{thm_vol_form} immediately implies the following result.
	\begin{cor}
		For any effective holomorphic line bundle $L$, the following holds
		\begin{equation}\label{eq_form_upp_bnd}
			\kappa(L) \leq \kappa_{{\rm{np}}}(L), \qquad {\rm{vol}}_{\kappa}(L) \leq {\rm{vol}}_{\kappa, {\rm{np}}}(L).
		\end{equation}
	\end{cor}
	\begin{rem}
		When $L$ is big, both inequalities in (\ref{eq_form_upp_bnd}) become equalities by (\ref{eq_vol_bouck}); however, in general, the reverse inequalities in (\ref{eq_form_upp_bnd}) do not hold.
		Indeed, if the reverse inequality for the Kodaira-Iitaka dimension were to hold, then -- since the right-hand side of (\ref{eq_form_upp_bnd}) depends solely on the numerical data, namely the class $c_1(L)$ -- it would follow that the Kodaira-Iitaka dimension is a numerical invariant of $L$. 
		This, however, is false, see Lehmann \cite[Example 6.1]{LehmannIitaka}.
		Consequently, (\ref{eq_env_ident}) also fails in general, i.e. when the line bundle $L$ is not big.
	\end{rem}
	\par
	\begin{sloppypar} 
	The study of the various numerical incarnations of the Kodaira-Iitaka dimension has recently drawn significant interest.
	Several equivalent and non-equivalent definitions have been proposed in the literature; see, for instance, \cite{NakayamaZariski, BDPP, LehmannIitaka, LesieutreIitaka, ChoiParknumII, FLTVolume}.
	Among the existing notions, it appears that the one defined via the movable intersection product in \cite{BDPP}, which we denote by $\kappa_{\mathrm{mov}}(L)$, is currently the smallest one; see \cite[Theorem 1.1 and after]{ChoiParknumII}.
	This invariant is defined as follows
	\begin{equation}
		\kappa_{{\rm{mov}}}(L) = \max\Big\{
			i \in \{ 0, 1, \ldots, n \} :  \langle c_1(L)^{i} \rangle \neq 0
		\Big\},
	\end{equation}
	where $\langle \cdot \rangle$ is the movable intersection product, see \cite[Definition 3.6]{BDPP} for details.
	Note that the non-pluripolar product is known to be dominated by the movable intersection product, see \cite[Proposition 1.20]{BEGZ}, and immediately from this, we get the following inequality
	\begin{equation}\label{eq_np_bdpp_compar}
		\kappa_{{\rm{np}}}(L) \leq  \kappa_{{\rm{mov}}}(L).
	\end{equation}
	In general, the inverse inequality in (\ref{eq_np_bdpp_compar}) fails. 
	Indeed, in \cite[Example 1.7]{DPSNef} the authors constructed a nef line bundle $L$ over a ruled surface for which $c_1(L)$ is non-trivial (and hence $\kappa_{{\rm{mov}}}(L) \geq 1$, cf. \cite[p. 219]{BEGZ}), and for which there is a unique closed positive $(1, 1)$-current $T$ in the class $c_1(L)$. 
	The current is given by the integration along a certain curve, and so $\kappa_{{\rm{np}}}(L) = 0$.
	\par 
	Providing an algebraic interpretation of $\kappa_{\mathrm{np}}(L)$, in the spirit of those for $\kappa_{\mathrm{mov}}(L)$ given in \cite{LehmannIitaka}, is an interesting question that we do not address in this article.
	\end{sloppypar} 
	\par 	
	Our next result concerns the refinement of (\ref{eq_vol_form}) on the metric level, computing the asymptotic ratio of the volumes of unit balls defined by the sup-norms on the linear series.
	\par
	For further purposes, it becomes necessary to also consider the sup-norms calculated over subsets.
	We hence fix a compact subset $K$ which we assume to be \textit{non-pluripolar}, i.e. not contained in the $\{-\infty\}$-locus of some plurisubharmonic (\textit{psh}) function.
	For a continuous metric $h^L$ on $L$, we define the (singular) metric $P[W, K, h^L]$ analogously to (\ref{eq_env_w}) except that instead of the sup-norms ${\textrm{Ban}}_k^{\infty}[W](h^L)$, we consider the sup-norms over $K$, which we denote by ${\textrm{Ban}}_k^{\infty}[W](K, h^L)$.
	As $K$ is non-pluripolar, $P[W, K, h^L]$ has a psh potential, cf. \cite[Theorem 9.17]{GuedjZeriahBook}.
	\par 
	For continuous metrics $h^L_0$, $h^L_1$ on $L$, we define the relative Monge-Ampère $\kappa$-energy as
	\begin{multline}\label{eq_diff_energy}
		\mathscr{E}_{\kappa}(P[W, K, h^L_0]) - \mathscr{E}_{\kappa}(P[W, K, h^L_1])
		\\
		=
		\frac{1}{2(\kappa(W) + 1) \cdot \varphi_{*} [\omega]^{n - \kappa(W)}}
		\sum_{i = 0}^{\kappa(W)} \int \log\Big( \frac{P[W, K, h^L_1]}{P[W, K, h^L_0]} \Big)
		\cdot
		\\
		 \cdot c_1(L, P[W, K, h^L_0])^i \wedge c_1(L, P[W, K, h^L_1])^{\kappa(W) - i} \wedge \omega^{n - \kappa(W)}.
	\end{multline}
	In Remark \ref{rem_ma_kappa_ma_usual}, we prove that (\ref{eq_diff_energy}) is independent of the choice of $\omega$, and that it satisfies the cocycle property; therefore, it is natural to regard it as a difference.
	\par 
	For every $k \in \nat$, we fix a Hermitian norm $H_k$ on $W_k$, which allows us to calculate the volumes $\vol(\cdot)$ of measurable subsets in $W_k$. 
	Note that while such volumes depend on the choice of $H_k$, their ratio does not.
	For $i = 0, 1$, we denote by $\mathbb{B}_k[W, K, h^L_i]$ the unit balls in $W_k$ corresponding to the norms ${\textrm{Ban}}_k^{\infty}[W](K, h^L)$.
	Our next result expresses the ratio of the volumes of $\mathbb{B}_k[W, K, h^L_i]$ in terms of their relative Monge-Ampère $\kappa$-energy as follows.
	\begin{thm}\label{thm_volumes_ratio}
		For any continuous metrics $h^L_0, h^L_1$ on $L$, we have
		\begin{equation}\label{rem_thm_isom2}
			\lim_{k \to \infty}
			\frac{1}{k^{\kappa(W) + 1}} \cdot \log \Big( \frac{\vol(\mathbb{B}_k[W, K, h^L_0])}{\vol(\mathbb{B}_k[W, K, h^L_1])} \Big)
			=
			\mathscr{E}_{\kappa}(P[W, K, h^L_0]) - \mathscr{E}_{\kappa}(P[W, K, h^L_1]).
		\end{equation}
	\end{thm}
	\begin{rem}
		a) Berman-Boucksom in \cite{BermanBouckBalls} established a version Theorem \ref{thm_volumes_ratio} for $W = R(X, L)$ and $L$ big.
		In \cite{BermanBouckBalls}, the relative Monge-Ampère $\kappa$-energy is replaced by the relative Monge-Ampère energy, see (\ref{eq_diff_energy_a}), and the envelope $P[W, K, h^L]$ is replaced with $P[K, h^L]$, which is defined as in (\ref{eq_sic_env}), but with the inequality $h^L_0 \geq h^L$ required to hold only over $K$. 
		Our results are compatible, as similarly to (\ref{eq_env_ident}), for big $L$ and $W = R(X, L)$, we have 
		\begin{equation}\label{eq_env_ident2}
    		P[W, K, h^L] = P[K, h^L].
		\end{equation}
		Note that our proof of Theorem \ref{thm_volumes_ratio} relies on \cite{BermanBouckBalls}.
		\par 
		b) The reader will check that if $h^L_0 = 2 h^L_1$, then Theorem \ref{thm_volumes_ratio} yields exactly (\ref{eq_vol_form}), which explains why we call Theorem \ref{thm_volumes_ratio} a metric refinement of (\ref{eq_vol_form}).
	\end{rem}
	\par
	We now turn our attention to a pointwise refinement of Theorem \ref{thm_vol_form}, which we formulate in the language of the partial Bergman kernels.
	We fix a continuous metric $h^L$ on $L$ and a positive Borel measure $\mu$ with a support equal to a compact non-pluripolar subset $K \subset X$.
	We denote by ${\textrm{Hilb}}_k[W](h^L, \mu)$ the positive semi-definite form on $W_k$ defined for $s_1, s_2 \in W_k$ as follows
	\begin{equation}\label{eq_defn_l2}
			\langle s_1, s_2 \rangle_{{\textrm{Hilb}}_k[W](h^L, \mu)} = \int_X \langle s_1(x), s_2(x) \rangle_{(h^L)^{k}} \cdot d \mu(x).
	\end{equation}
	Remark that since $K$ is non-pluripolar, the above form is positive definite.
	We say that a measure $\mu$ is \textit{Bernstein-Markov} with respect to $W$ and $(K, h^L)$, if for each $\epsilon > 0$, there is $k_0 \in \nat$, so that
	\begin{equation}\label{eq_bm_clas}
		{\textrm{Ban}}_k^{\infty}[W](K, h^L)
		\leq
		C
		\cdot
		\exp(\epsilon k)
		\cdot
		{\textrm{Hilb}}_k[W](h^L, \mu), \quad \text{for any } k \geq k_0.
	\end{equation}
	\par 
	For surveys on the Bernstein-Markov property, see \cite{BernsteinMarkovSurvey, MarinescuVuSurvey}.  
	Note that the Lebesgue measure on a smoothly bounded domain in $X$, or on a totally real compact submanifold of real dimension $n$, is Bernstein-Markov, see \cite[Corollary 1.8]{BerBoucNys}, \cite[Proposition 3.6]{BernsteinMarkovSurvey} and \cite[Theorem 1.3]{MarinescuVu}.
	\par
	We define a sequence of measures on $X$ as follows
	\begin{equation}\label{eq_mu_begmm}
		\mu_k[W, \mu, h^L] := \frac{1}{k^{\kappa(W)}} B_k[W, \mu, h^L](x, x) \cdot d \mu(x),
	\end{equation}
	where $B_k[W, \mu, h^L](x, x) \geq 0$ is the partial Bergman (or Christoffel-Darboux) kernel, defined as
	\begin{equation}\label{eq_bergm_kern}
		B_k[W, \mu, h^L](x, y) := \sum_{i = 1}^{m_k} s_i(x) \cdot s_i(y)^* \in L_x^{\otimes k} \otimes (L_y^{\otimes k})^*,
	\end{equation}
	where $s_i$, $i = 1, \ldots, m_k$ is an orthonormal basis of $(W_k, {\textrm{Hilb}}_k[W](h^L, \mu))$ and $m_k := \dim W_k$. 
	One can easily verify that (\ref{eq_bergm_kern}) doesn't depend on the choice of the basis.
	\par 
	It was established by Berman-Boucksom-Witt Nyström \cite{BerBoucNys} that when $W = R(X, L)$ and $L$ is a big line bundle, $\mu_k[W, \mu, h^L]$ converge weakly towards the equilibrium measure of $(K, h^L)$; see also Bloom-Levenberg \cite{BloomLeven1}, \cite{BloomLeven2} for the case of projective spaces.
	\par 
	When $X$ admits an action of a compact connected Lie group $G$ which can be lifted to a holomorphic action on $L$, we can consider the graded linear series $W_k := H^0(X,L^{\otimes k})^G$ of $G$-invariant sections. 
	When $L$ is ample, $h^L$ has a positive curvature, the $G$-action preserves $h^L$, and the measure $\mu$ is the symplectic volume form associated with $c_1(L, h^L)$, Ma-Zhang in \cite{MaZhBKSR} established that $\mu_k[W, \mu, h^L]$ converge weakly, as $k \to \infty$, to an explicit measure supported on the zero set of the associated moment map, provided that the latter has $0 \in \mathfrak{g}^*$ as a regular value.
	\par 
	From these two examples, together with many others reviewed in Section \ref{sect_exampl}, it is natural to expect that the weak convergence holds in general.
	However, this expectation is false: in Section \ref{sect_no_weak}, we construct an explicit (and easy) counterexample.
	The absence of the weak convergence is closely related to the fact that the relative Monge-Ampère $\kappa$-energy is, in general, \textit{not differentiable}, as we will explain in Section \ref{sect_no_weak}. 
	This contrasts with the usual relative Monge-Ampère energy in the context of a big line bundle and a complete linear series, where the differentiability was established by Berman-Boucksom \cite{BermanBouckBalls}.
	\par 
	Nevertheless, as shown below, the weak convergence does hold for birational $W$.
	Moreover, even without any assumptions, a weaker form of the weak convergence still holds.
	To state our result, we fix a complex analytic space $Z$ and a holomorphic map $\rho : X \to Z$, which we assume to factorize, for $m \in \nat$ divisible enough, through the rational maps of the linear series as
	\begin{equation}\label{eq_rat_maps_z}
		\begin{tikzcd}
		& X \arrow[dl, "\rho"'] \arrow[dr, dashed, "\varphi_m"] & \\
Z & & \arrow[ll, dashed, "\psi_m"'] Y_m.
		\end{tikzcd}
	\end{equation}
	Of course, by the discussion after (\ref{eq_ki_maps}), it suffices to verify the existence of such factorization for a single $m$ large enough.
	It also follows from (\ref{eq_ym_ykm}) that one can consider $Z = Y_m$, modulo replacing $X$ by its birational model.
	In Remark \ref{rem_indep_mueq}, we show that the following measure on $Z$ does not depend on the choice of $\omega$:
	\begin{equation}\label{eq_quil_pushed_y}
		\mu_{\mathrm{eq}}[W, Z, K, h^L] :=
		\frac{1}{\varphi_{*} [\omega]^{n - \kappa(W)}} \rho_{*}  
		\Big(
		c_1(L, P[W, K, h^L])^{\kappa(W)} \wedge \omega^{n - \kappa(W)}
		\Big).
	\end{equation}
	We can now state our final result.
	\begin{thm}\label{thm_part_bk_get}
		If the measure $\mu$ is Bernstein-Markov with respect to $W$ and $(K, h^L)$, then, as $k \to \infty$, the sequence of measures $\rho_{*}( \mu_k[W, \mu, h^L])$ on $Z$ converges weakly towards $\mu_{\mathrm{eq}}[W, Z, K, h^L]$.
	\end{thm}
	\begin{rem}
		a)
		Integrating $\rho_{*}(\mu_k[W, \mu, h^L])$ over $Z$, Theorem \ref{thm_part_bk_get} yields (\ref{eq_vol_form}), thereby justifying the interpretation of Theorem \ref{thm_part_bk_get} as a pointwise analogue of Theorem \ref{thm_vol_form}.
		\par 
		b) 
		If $W$ is birational, then one can take $Z := X$ in (\ref{eq_rat_maps_z}).
		Theorem \ref{thm_part_bk_get} then shows that $\mu_k[W, \mu, h^L]$ converge weakly, as $k \to \infty$, towards $c_1(L, P[W, K, h^L])^{n}$. 
		When $W = R(X, L)$ and $L$ is big, by (\ref{eq_env_ident}), this is compatible with the aforementioned result from \cite{BerBoucNys}, establishing that in this case the limiting measure is given by $c_1(L, P[K, h^L])^n$.
	\end{rem}
	\par 
	This article is organized as follows.
	In Section \ref{sect_exampl}, we revisit several natural examples of linear series, place the results of the present article in the context of the existing work and state some applications.
	In Section \ref{sect_2}, we establish Theorems \ref{thm_ki_form} and \ref{thm_vol_form}.
	In Section \ref{sect_vol_balls_compar}, we prove Theorem \ref{thm_volumes_ratio}.
	In Section \ref{sect_pbk}, we establish Theorem \ref{thm_part_bk_get}, and relate it with the differentiability of the relative Monge-Ampère $\kappa$-energy.
	Finally, in Section \ref{sect_lin_ser_sing_type} we discuss some applications.
	\par
	\textbf{Notation.}
	We denote by $d = \partial + \dbar$ the usual decomposition of the exterior derivative in terms of its $(1, 0)$ and $(0, 1)$ parts, and we set 
	\begin{equation}
		d^c := \frac{\partial - \dbar}{2 \pi \imun}.	
	\end{equation}
	The Poincaré-Lelong formula gives us the following: for any smooth metric $h^L$ on a line bundle $L$, any $s \in H^0(X, L)$, $s \neq 0$, for the divisor $E := [s = 0]$, we have
	\begin{equation}\label{eq_poinc_lll}
		c_1(L, h^L) + \ddc \log |s(x)|_{h^L} = [E],
	\end{equation}
	where $c_1(L, h^L)$ is the first Chern form of $(L, h^L)$, defined as $\frac{\imun}{2 \pi} R^L$, where $R^L$ is the curvature of the Chern connection and $[E]$ is the current of integration along $E$.
	Note also that for an arbitrary smooth function $\phi: X \to \real$, we have
	\begin{equation}\label{eq_c1_bc}
		c_1(L, h^L \cdot \exp(-2 \phi))
		=
		c_1(L, h^L)
		+
		\ddc \phi.
	\end{equation}
	\par 
	We say that $\phi$ is a potential of a metric $h^L$ on $L$ with respect to a smooth metric $h^L_0$ if 
	\begin{equation}\label{eq_hl_hl0_phi_com}
		h^L = h^L_0 \cdot \exp(-2 \phi).
	\end{equation}
	Occasionally, we will refer to a (local) potential of $h^L$, which is a potential with respect to a metric that trivializes a particular local holomorphic frame.
	\par
	We also consider semimetrics (resp. singular metrics), which are the objects $h^L$ defined by the expression (\ref{eq_hl_hl0_phi_com}), so that $\phi$ is allowed to take values in $\real \cup \{ +\infty \}$ (resp. $\real \cup \{ -\infty \}$).
	When $\phi$ lies in $L^1_{loc}$, through (\ref{eq_c1_bc}), we extend the definition of $c_1(L, h^L)$ as a $(1, 1)$-current through (\ref{eq_c1_bc}).
	\par 
	For a function $f: X \to [-\infty, +\infty]$, we denote by $f^*$ its upper semicontinuous regularization, defined for any $x \in X$ as 
	\begin{equation}
		f^*(x) := \lim_{\epsilon \to 0} \sup_{B(x, \epsilon)} f,
	\end{equation}
	where $B(x, \epsilon) \subset X$ is a ball of radius $\epsilon$ around $x$ (with respect to some fixed Riemannian metric). 
	The lower semicontinuous regularization of a metric $h^L$, which we denote by $h^L_*$, is obtained by replacing the potential of $h^L$ with its upper semicontinuous regularization.
	\par 
	By a canonical singular metric $h^E_{{\rm{sing}}}$ on a line bundle $\mathscr{O}(E)$ associated with a divisor $E$ we mean a singular metric, defined such that for any $x \in X \setminus E$, we have $|s_E(x)|_{h^E_{{\rm{sing}}}} = 1$, where $s_E$ is the \textit{canonical holomorphic section} of $\mathscr{O}(E)$, verifying ${\rm{div}}(s_E) = E$.
	\par 	
	A function $\phi: X \to [-\infty, +\infty[$ is called quasi-plurisubharmonic (qpsh) if it can be locally written as the sum of a plurisubharmonic function and a smooth function. 
	For a smooth closed $(1, 1)$-form $\alpha$, we say $\phi$ is $\alpha$-plurisubharmonic ($\alpha$-psh) if it is qpsh and $\alpha + \ddc \phi > 0$ in the sense of currents.
	We let $\psh(X, \alpha)$ denote the set of $\alpha$-psh functions that are not identically $-\infty$.
	\par 
	When the class $[\alpha] \in H^2(X, \real) \cap H^{1, 1}(X)$ contains a positive closed $(1, 1)$-current, we say $[\alpha]$ is psef.
	If it moreover contains a Kähler current, such a class is called \textit{big}. 
	\par 
	A norm $N_V = \| \cdot \|_V$ on a finite-dimensional vector space $V$ naturally induces the norm $\| \cdot \|_Q$ on any quotient $Q$, $\pi : V \to Q$ of $V$ through
	\begin{equation}\label{eq_defn_quot_norm}
		\| f \|_Q
		:=
		\inf \big \{
		 \| g \|_V
		 ;
		 \quad
		 g \in V, 
		 \pi(g) = f
		\},
		\qquad f \in Q.
	\end{equation}
	For every inclusion $\iota: E \to V$, we also define a norm $\iota^* N_V$ on $E$ in a natural way.
	\par
	\textbf{Acknowledgement.}
	 This work was supported by the CNRS, École Polytechnique, and in part by the ANR projects QCM (ANR-23-CE40-0021-01), AdAnAr (ANR-24-CE40-6184), STENTOR (ANR-24-CE40-5905-01) and CanQuantFilt (ANR-25-ERCS-0009).
	 
	\section{Linear series: a case study}\label{sect_exampl}
	The aim of this section is to revisit several natural examples of linear series, to place the results of the present article in the context of the existing work and to state some applications.
	\par 
	Remark first that if we have two linear series $W^1 = \oplus_{k = 0}^{\infty} W^1_k$, $W^2 = \oplus_{k = 0}^{\infty} W^2_k$, so that $W^1 \subset W^2$, then the respective rational maps relate as in the following diagram
	\begin{equation}\label{eq_lin_ser_emb}
		\begin{tikzcd}
		& X \arrow[dl, dashed] \arrow[dr, dashed] & \\
\mathbb{P}((W_k^1)^*) & & \arrow[ll, dashed] \mathbb{P}((W_k^2)^*).
		\end{tikzcd}
	\end{equation}
	In particular, if $W^1$ is birational, then $W^2$ is birational too.
	\par 
	We say that a linear series $W$ \textit{contains an ample series}, cf. \cite{LazMus}, if there exist an ample line bundle $A$ over $X$, an effective divisor $E$ on $X$ such that $L = A \otimes \mathscr{O}(E)$, and the inclusion $H^0(X, A^{\otimes k}) \hookrightarrow W_k$, which is compatible with the natural map $H^0(X, A^{\otimes k}) \hookrightarrow H^0(X, L^{\otimes k})$ associated with the embedding $R(X, A) \subset R(X, L)$.
	We then have the diagram analogous to (\ref{eq_lin_ser_emb}) for the rational maps, and it follows from the Kodaira embedding theorem that any linear series containing an ample series is automatically birational.
	\par 
	For the latter use, we introduce the \textit{base locus} of a holomorphic line bundle $L'$ on $X$ as
	\begin{equation}
		Bs(L') := \Big\{ x \in X : \text{there is } s \in H^0(X, L'), \text{ such that } s(x) \neq 0 \Big\}.
	\end{equation}
	One can analogously define the base locus of a linear series.
	Moreover, by locally trivializing the sections from the linear series, we can also define the corresponding ideal sheaf. 
	\par 
	The \textit{stable base locus} of a line bundle $L'$ is defined as
	\begin{equation}
		\mathbb{B}(L') := \cap_{k = 1}^{+\infty} Bs((L')^{\otimes k}).
	\end{equation}
	The notion of the stable base locus extends naturally to $\mathbb{Q}$-line bundles, a fact we use to define the \textit{augmented base locus} as
	\begin{equation}\label{eq_aug_bl}
		\mathbb{B}_+(L') := \mathbb{B}(L' - \epsilon A),
	\end{equation}
	where $A$ is an arbitrary ample line bundle and $\epsilon$ is a sufficiently small rational number.
	From Noether's finiteness result, the above definition doesn't depend on $\epsilon > 0$ as long as it is very small.
	From this, it is immediate to see that it doesn't depend on the choice of the ample line bundle $A$.
	\par 
	\textbf{Example 0: the complete linear series.}
	Let $X$ be a projective manifold and $L$ be a line bundle over $X$. 
	Define the complete linear series as $W = R(X, L)$.
	\par 
	By definition, both the Kodaira-Iitaka dimension and the multiplicity of $W$ agree with the corresponding invariants of $L$.
	Fujita's decomposition, \cite[Corollary 2.2.7]{LazarBookI}, essentially says that when $L$ is big, $R(X, L)$ contains an ample series; in particular, it is birational by the discussion after (\ref{eq_lin_ser_emb}).
	Inversely, by (\ref{eq_dim_ki_yk}), we see that if the complete linear series is birational, then $L$ is big.
	\par 
	Let us now fix $x \not\in \mathbb{B}_+(L)$.
	Then there is an ample $\mathbb{Q}$-line bundle $A$, a sufficiently large $r \in \nat^*$ and $s \in H^0(X, (L \otimes A^{-1})^{\otimes r})$, so that $s(x) \neq 0$.
	By (\ref{eq_lin_ser_emb}) and the Kodaira embedding theorem, we deduce that for $m$ large enough, the point $x$ lies in the locus where the associated rational map $X \dashrightarrow Y_m$ is well-defined and is an isomorphism, cf. \cite[Lemma 2.16]{LazMus}.
	\par 
	The study of the Bergman kernel (i.e., the partial Bergman kernel associated with the complete linear series) has been a central topic in complex analysis and geometry over the past several decades, see \cite{TianBerg, ZeldBerg, Bouche, Caltin, DaiLiuMa, MaHol, BermanEnvProj, ComanMaMarMoish} for some important developments.
	\par 
	\textbf{Example 1: finitely generated linear series.}
	Let $V \subset R(X, L)$ be a graded finite-dimensional subspace. 
	The graded subalgebra of $R(X, L)$ induced by $V$ yields a linear series.
	When $V$ is homogeneous of degree $1$, the respective rational maps are given by $X \dashrightarrow \mathbb{P}(V^*)$. The Hilbert-Samuel theorem and (\ref{eq_sym_yk_rel}) imply, cf. \cite[Theorem 3.1]{KavehKhov}, that the Kodaira-Iitaka dimension and the multiplicity of the induced series coincide, respectively, with the dimension and the degree of the image of the Kodaira-Iitaka map in the projective space $\mathbb{P}(V^*)$.
	\par 
	\textbf{Example 2: linear series induced by holomorphic maps.}
	Let $\pi : X \to Y$ be a holomorphic map between two complex projective manifolds and $L_Y$ be a line bundle on $Y$.
	Any linear series $W$ on $(Y, L_Y)$ induces a linear series $\pi^* W$ on $R(X, \pi^* L_Y)$ by the pull-back.
	The respective rational maps factorize as in the following diagram
	\begin{equation}\label{eq_lin_ser_pull}
		\begin{tikzcd}
		X \arrow[r, "\pi"] \arrow[d, dashed] & Y \arrow[d, dashed] \\
\mathbb{P}(\pi^* W_k^*)  \arrow[r, hook] & \mathbb{P}(W_k^*).
		\end{tikzcd}
	\end{equation}
	Any holomorphic map can be decomposed as a composition of a surjective projection and an embedding, and it is convenient to separate these two cases in the discussion of the above construction.
	\par 
	If $\pi$ is surjective, the map from $W_k$ to $\pi^* W_k$ is an isomorphism, and the above procedure does not change the Kodaira-Iitaka dimension and the multiplicity.
	The class of birational linear series is preserved only if $\pi$ is itself birational by (\ref{eq_lin_ser_pull}). In particular, if $\pi$ is a non-trivial finite covering, the linear series $\pi^* W$ is not birational.
	Note, however, that if $\pi$ is birational, a theorem of Zariski, \cite[Corollary III.11.4]{HartsBook}, implies that the pull-back induces an isomorphism between $R(Y, L_Y)$ and $R(X, \pi^* L_Y)$, limiting the potential use of the above transformation.
	\par 
	When $Y$ is an irreducible variety instead of a manifold, one may apply this construction to a resolution of singularities, and therefore reduce the study of linear series on irreducible varieties to the corresponding study on manifolds (this shows in particular that the results of the present article extend immediately to the singular setting).
	\par 
	When $\pi$ is an embedding, and the initial linear series is the complete linear series, the above construction yields the so-called \textit{restricted linear series}, defined as
	\begin{equation}
		W_k := \Im (H^0(Y, L^{\otimes k}) \to H^0(X, \pi^* L^{\otimes l})).
	\end{equation}
	When $L_Y$ is big and $X$ lies outside of the augmented base locus, the resulting linear series in $R(X, \pi^* L)$ is birational by (\ref{eq_lin_ser_pull}), as the Kodaira-Iitaka map of $(Y, L_Y)$ is an isomorphism outside of the augmented base locus by the discussion after (\ref{eq_lin_ser_emb}), cf. \cite[Lemma 2.16]{LazMus}.
	In this setting, Theorem \ref{thm_vol_form} was proved in \cite{MatsumuraRestr} and \cite{HisamResBerg}, and versions of Theorem \ref{thm_part_bk_get} appeared in \cite{HisamResBerg} and \cite{ComanMarinNguyenRest}.
	\par 
	\textbf{Example 3: linear series shifted by a divisor.}
	We fix an arbitrary effective divisor $D$ on $X$, and a linear series $W = \oplus_{k = 0}^{\infty} W_k \subset R(X, L)$.
	The vector spaces $W_k^D$ obtained by a multiplication of $W_k$ by the $k$-th tensor power of the canonical holomorphic section of $\mathscr{O}(D)$ define a linear series in $R(X, L \otimes \mathscr{O}(D))$.
	Of course, this procedure does not change the Kodaira-Iitaka dimension and the multiplicity.
	It also preserves the class of birational linear series by the argument as in (\ref{eq_lin_ser_emb}).
	The study of linear series on arbitrary line bundles over a projective manifold can be reduced to the corresponding study on ample line bundles.
	\par 
	By combining Examples 2 and 3, we see in particular that any Fujita approximation for $L$ yields a birational linear series, see \cite[Definition 11.4.3 and Theorem 11.4.4]{LazarBookII}.
	\par 
	\textbf{Example 4: filtration-based linear series.}
	Recall that a \textit{decreasing $\real$-filtration} $\mathcal{F}$ of a vector space $V$ is a map from $\real$ to vector subspaces of $V$, $t \mapsto \mathcal{F}^t V$, verifying $\mathcal{F}^t V \subset \mathcal{F}^s V$ for $t > s$.
	A filtration $\mathcal{F}$ on $R(X, L)$ is a collection $(\mathcal{F}_k)_{k = 0}^{\infty}$ of decreasing filtrations $\mathcal{F}_k$ on $H^0(X, L^{\otimes k})$.
	We say that $\mathcal{F}$ is \textit{submultiplicative} if for any $t, s \in \real$, $k, l \in \nat$ we have 
	\begin{equation}\label{eq_sumb_filt}
		\mathcal{F}^t_k H^0(X, L^{\otimes k}) \cdot \mathcal{F}^s_l H^0(X, L^{\otimes l}) \subset \mathcal{F}^{t + s}_{k + l} H^0(X, L^{\otimes (k + l)}).
	\end{equation}
	We say that $\mathcal{F}$ is bounded if there is $C > 0$ so that for any $k \in \nat^*$, we have $\mathcal{F}^{Ck}_k H^0(X, L^{\otimes k}) = \{0\}$.
	\par 
	Let us now fix a bounded submultiplicative filtration $\mathcal{F}$.
	Clearly, for any $t \in \real$, the subspace $W_k := \mathcal{F}^{tk}_k H^0(X, L^{\otimes k})$ defines a graded linear series.
	Boucksom-Chen in \cite[Lemma 1.6]{BouckChen} verified that as long as $t$ does not attain a certain critical value $\lambda$ (which is characterized as the minimal $\lambda \in \real$ so that for any $\epsilon > 0$, we have $\dim \mathcal{F}^{(\lambda + \epsilon)k}_k H^0(X, L^{\otimes k}) = o(k^n)$), the resulting linear series contains an ample series.
	It is, hence, birational by the discussion following (\ref{eq_lin_ser_emb}).
	\par 
	The study of the filtration-based linear series is tightly related with the construction of the geodesic ray associated with the filtration, see \cite{PhongSturmTestGeodK, RossNystAnalTConf, HisamSpecMeas}. 
	In the context of geometric quantization, it has been recently studied in \cite{FinNarSim, FinSubmToepl, FinGQBig}.
	\par 
	\textbf{Example 5: linear series associated with an ideal sheaf.}
	Let us consider an ideal sheaf $\mathcal{I}$ in $\mathscr{O}_X$.
	Then $W_k := H^0(X, L^{\otimes k} \otimes \mathcal{I}^k)$ defines a linear series in $R(X, L)$.
	Of course, this provides a special case of the previous construction, if one defines the corresponding filtration as $\mathcal{F}^t_k H^0(X, L^{\otimes k}) := H^0(X, L^{\otimes k} \otimes \mathcal{I}^{\lceil t \rceil})$.
	\par 
	When the sheaf is given by the sheaf of holomorphic functions vanishing along a certain subvariety, the partial Bergman kernels have been studied in \cite{BermanEnvProj, ComMar, RossSingPartDens}.
	\par 
	A corresponding study when the power of the ideal sheaf is replaced by a multiplier ideal sheaf associated with a psh function has been done in \cite{DarvXiaVolumes}.
	Note, however, that the latter construction does not generally provide a linear series.
	\par 
	\textbf{Example 6: valuation-based linear series.}
	Let $v$ be a real valuation on the field or meromorphic function on $X$, $K(X)$.
	Clearly, it defines a valuation on $R(X, L)$.
	For any $t \in \real$, the subspace $W_k := \{\, s \in H^0(X, L^{\otimes k}) : v(s) \geq tk \}$ defines a graded linear series.
	Of course, it is a special case of the filtration-based linear series.
	\par
	For example, if $(X, L)$ carries a $\comp^*$-action, it induces the $\comp^*$-action on $H^0(X, L^{\otimes k})$, which hence admits a weight decomposition. 
	Then the direct sum of the vector spaces generated by the elements of the weight in $[tk, +\infty[$ defines a graded linear series, and the associated partial Bergman kernels has been studied in \cite{IoosPartial}.
	\par 
	\textbf{Example 7: product of two linear series.}
	Given two linear series $W^1, W^2$ in $R(X, L)$, one can define their product, $W^1 \cdot W^2$ as the smallest linear series containing all the products. 
	Immediately from the definitions, we see that if both $W^1, W^2$ are non-empty, then the Kodaira-Iitaka dimension of $W^1 \cdot W^2$ is no smaller that the maximum of the Kodaira-Iitaka dimensions of $W^1$ and $W^2$.
	When all three dimensions coincide, the respective multiplicities verify a Brunn-Minkowski type inequality, see \cite[Theorem 2.33]{KavehKhov}.
	Analogously to (\ref{eq_lin_ser_emb}), it is easy to see that if both $W^1, W^2$ are non-empty and at least one is birational, then $W^1 \cdot W^2$ is birational.
	\par 
	\textbf{Example 8: integral closure of a linear series.}
	Given a linear series $W$ in $R(X, L)$, one can define its integral closure, $\overline{W} = \oplus_{k = 0}^{+\infty} \overline{W}_k$, as follows
	\begin{multline}
		\overline{W}_k = \Big\{ s \in H^0(X, L^{\otimes k}) 
		\\
		: \text{there exists a monic polynomial } f(t) \in W[t] \text{ such that } f(s) = 0 
		\Big\}.
	\end{multline}
	It is then standard that $\overline{W}$ defined this way is indeed a linear series, see \cite[Theorem 4.2]{EisenbudBook} and \cite[Theorem 2.32]{SwanHuneke}.
	We denote by $Y_m$ and $\overline{Y}_m$ the closures of the images of $X$ with respect to the rational maps associated with $W_m$ and $\overline{W}_m$.
	Then by (\ref{eq_lin_ser_emb}), we have a dominant rational map
	\begin{equation}\label{eq_over_y_ym_rat}
		\overline{Y}_m \dashrightarrow Y_m.
	\end{equation}
	The first part of Theorem \ref{thm_integ_closuse} below with (\ref{eq_dim_ki_yk}) show that for $m$ large enough, we have $\dim(\overline{Y}_m) = \dim(Y_m)$.
	In particular, the map (\ref{eq_over_y_ym_rat}) is generically finite for $m$ large enough.
	We denote by ${\rm{deg}}(\overline{W} : W)$ the generic degree of this map (by a discussion after (\ref{eq_ki_maps}), it does not depend on $m$ as long as $m$ is large enough).
	In Section \ref{sect_lin_ser_sing_type}, building on Theorems \ref{thm_ki_form}, \ref{thm_vol_form}, we prove the following result.
	\begin{thm}\label{thm_integ_closuse}
		For any linear series $W$, $\kappa(\overline{W}) = \kappa(W)$ and ${\rm{vol}}_{\kappa}(\overline{W}) = {\rm{deg}}(\overline{W} : W) \cdot {\rm{vol}}_{\kappa}(W)$.
	\end{thm}
	\par 
	Note that if $W$ is birational, $\overline{W}$ stays birational by (\ref{eq_lin_ser_emb}), and then ${\rm{deg}}(\overline{W} : W) = 1$. 
	But it is easy to construct non-birational $W$ which have birational $\overline{W}$ (consider the linear series of polynomials of even degree inside of the ring of polynomials), and then ${\rm{deg}}(\overline{W} : W) > 1$.
	\par 
	\textbf{Example 9: invariant linear series.}
	When $L$ is ample, $X$ admits an action of a group $G$ which can be lifted to a holomorphic action on $L$, we can consider the $G$-invariant linear series $W_k := H^0(X,L^{\otimes k})^G$.
	If $G$ is a reductive Lie group, then the ring $W = \oplus_{k = 0}^{+ \infty} W_k$ is finitely generated, cf. \cite[\S 1]{MumfordGIT}.
	We denote by $W^{(d)}$ the restriction of $W$ to $R(X, L^{\otimes d})$.
	If one takes $d$ sufficiently large so that $W^{(d)}$ is generated by $W_d$, then the space $X /\!/ G := {\rm{Proj}} (W^{(d)})$ carries a canonical line bundle $L_G$, so that for $k \in \nat$ large enough, we have $H^0(X /\!/ G, L_G^{\otimes k}) \simeq W_{kd}$, cf. \cite[Exercise II.5.14]{HartsBook}.
	Moreover, there is a natural rational map $\pi : X \dashrightarrow  X /\!/ G$ given by the composition of the Kodaira-Iitaka map and a rational projection $\mathbb{P}(H^0(X, L^{\otimes kd})) \dashrightarrow \mathbb{P}(W_{kd}^*)$.
	Consider a resolution of the indeterminacies of this map, carried out simultaneously with a resolution of the singularities of the base locus of $H^0(X, L^{\otimes d})^G$, providing the following diagram
	\begin{equation}\label{eq_resol_ind}
		\begin{tikzcd}
		& \widehat{X} \arrow[dl, "p"'] \arrow[dr, "\widehat{{\rm{\pi}}}"] & \\
		X \arrow[rr, dashed, "\pi"'] & & X /\!/ G.
		\end{tikzcd}
	\end{equation}
	Then for $\widehat{L} := p^* L$, there is an effective divisor $E$ on $\widehat{X}$ (given by the resolution of the base locus of $W_d$) and an isomorphism
	\begin{equation}
		\widehat{L}^{\otimes d}
		\simeq 
		\widehat{{\rm{\pi}}}^* L_G \otimes \mathscr{O}(E).
	\end{equation}
	Since the map $p: \widehat{X} \to X$ is birational, $W^{(d)}$ can be seen as a linear series in $R(\widehat{X}, \widehat{L})$ as described in Example 2.
	Under this identification, $W^{(d)}$ coincides in large degrees with the image of the complete linear series, $R(X /\!/ G, L_G)$, by the procedures given in Examples 2 and 3.
	Partial Bergman kernel associated with this linear series has been studied in \cite{MaZhBKSR}.
	\par 
	\textbf{Example 10: linear series associated with a singularity type.}
	We fix a smooth metric $h^L_0$ on a psef line bundle $L$, and let $\alpha = c_1(L, h^L_0)$.
	We fix a function $\phi \in \psh(X, \alpha)$ and consider the set
	\begin{equation}\label{eq_wkphi_defn}
		W_k(\phi) := \Big\{
			s \in H^0(X, L^{\otimes k}) : \frac{1}{k} \log |s|_{h^L_0} \preceq \phi
		\Big\},
	\end{equation}
	where $u \preceq v$ means that there is a constant $C \in \real$ such that $u \leq v + C$ everywhere on $X$.
	We claim, following Witt Nyst\"om \cite[\S 2.2]{WittNystromCanGrowth}, that this defines a linear series $W(\phi) = \oplus_{k = 0}^{+\infty} W_k(\phi)$.
	\par 
	It is easy to see that if $s_1 \in W_k(\phi)$ and $s_2 \in W_l(\phi)$, then $s_1 \cdot s_2 \in W_{k + l}(\phi)$.
	Let us now consider $s_1, s_2 \in W_k(\phi)$. 
	Let $C > 0$ be such that $\frac{1}{k} \log |s_1|,  \frac{1}{k} \log |s_2| \leq \phi + C$.
	Then
	\begin{equation}
		 \frac{1}{k} \log |s_1 + s_2| \leq \phi + C + \frac{\log(2)}{k},
	\end{equation}
	which shows that $s_1 + s_2 \in W_k(\phi)$.
	Altogether, this shows that $W(\phi)$ is a linear series.
	\par 
	Let us now relate the Kodaira-Iitaka dimension of $W(\phi)$ with the numerical dimension of the positive current $T := \alpha + dd^c \phi$, defined as
	\begin{equation}\label{eq_nd_defn}
		{\rm{nd}}(T)
		=
		\max\Big\{
				i \in \{ 0, 1, \ldots, n \} :  T^{i} \wedge \omega^{n - i} \neq 0
		\Big\},
	\end{equation}
	where $\omega$ is an arbitrary fixed K\"ahler form.
	Note that our definition (\ref{eq_nd_defn}) is non-standard \cite{DemaillyOnCohPSEF}, \cite{CaoNumDim}.
	We denote by $\varphi(\phi)_* [\omega]^{n - \kappa(W(\phi))}$ the integral of $\omega$ along a general fiber of the rational map associated with $W(\phi)$ in high degrees.
	The following result will be proved in Section \ref{sect_lin_ser_sing_type}.
	\begin{thm}\label{thm_kappa_wphi}
		The following inequality holds $\kappa(W(\phi)) \leq {\rm{nd}}(T)$.
		Moreover, we have
		\begin{equation}
			{\rm{vol}}_{\kappa}(W(\phi))
			\leq
			\frac{1}{\varphi(\phi)_{*} [\omega]^{n - \kappa(W(\phi))}} \int_X T^{\kappa(W(\phi))} \wedge \omega^{n - \kappa(W(\phi))}.
		\end{equation}
	\end{thm}
	\begin{rem}
		When $W(\phi)$ contains an ample series (and hence $\kappa(W(\phi)) = {\rm{nd}}(T) = n$ and $\varphi(\phi)_{*} [\omega]^{n - \kappa(W(\phi))} = 1$), Theorem \ref{thm_kappa_wphi} is due to Witt Nystr\"om \cite[Theorem 2.20]{WittNystromCanGrowth}.
	\end{rem}
	Based on the above construction, in Section \ref{sect_lin_ser_sing_type}, we introduce the analytic closure of a linear series.
	We relate it with the integral closure, and show that Theorem \ref{thm_integ_closuse} follows from it.

	\section{Analytic formulas for the Kodaira-Iitaka dimension and the multiplicity}\label{sect_2}
		
	The main goal of this section is to express the Kodaira-Iitaka dimension and the multiplicity of graded linear series in terms of the intersection theory of the plurisubharmonic envelope associated with the linear series.
	More precisely, in Sections \ref{sect_min_sing} and \ref{sect_fujita}, we recall some preliminairies from pluripotential theory and the asymptotic study of linear series respectively.
	In Section \ref{sect_fs_inter}, we recall some results concerning the Fubini-Study operator and establish Theorems \ref{thm_ki_form} and \ref{thm_vol_form}.

	\subsection{Preliminaries from pluripotential theory}\label{sect_min_sing}
	The main objective of this section is to review several essential definitions and results from pluripotential theory. 
	We begin by recalling the definition of potentials with minimal singularities, some facts concerning the envelope construction and then discuss the monotonicity formula, the continuity of the non-pluripolar product and the projection formula.
	\par 
	To define potentials with minimal singularities, we first fix a real smooth closed $(1, 1)$-form $\alpha$ in the psef class $[\alpha]$.
	Let us introduce a partial order on the space of $\alpha$-psh functions on $X$. 
	We say that a $\alpha$-psh function $\phi_0$ is \emph{more singular} than $\phi_1$ (and denote it by $\phi_0 \preceq \phi_1$) if there exists a constant $C > 0$ such that $\phi_0 \leq \phi_1 + C$.  
	Furthermore, we define the equivalence class of $\alpha$-psh functions associated with $\phi_0$ with respect to this order as 
	\begin{equation}
		[\phi_0] = \Big\{ \psi \in \psh(X, \alpha) : \phi_0 \preceq \psi \text{ and } \psi \preceq \phi_0 \Big\},
	\end{equation}
	and call it the \textit{singularity type} of $\phi_0$.
	\par 
	We extend the above order to currents in the same cohomology class by comparing their associated potentials, and we call currents that are minimal with respect to this order \textit{currents with minimal singularities}.
	As has been observed by Demailly, cf. \cite{DPSPseudoeff}, in a given psef cohomological class $[\alpha] \in H^{1, 1}(X) \cap H^2(X, \real)$, there is always a closed positive current with minimal singularities.
	To construct such a current, consider an arbitrary smooth form $\alpha$ from $[\alpha]$ and consider $P(\alpha) := \alpha + \ddc V_{\alpha}^*$, where $V_{\alpha}$ is defined as
	\begin{equation}\label{eq_v_theta}
		V_{\alpha} := \sup \big\{ \psi \in \psh(X, \alpha) : \psi \leq 0 \big\}.
	\end{equation}	 
	\par
	This latter construction can be translated in the language of metrics on line bundles.
	More precisely, let $L$ be a psef line bundle and $h^L$ be an arbitrary semimetric on $L$, which is bounded from below over a non-pluripolar subset $E$ by a continuous metric.
	Let us consider the envelope $P[h^L]$ defined as in (\ref{eq_sic_env}).
	Generalizing the observation before (\ref{eq_v_theta}), we have the following well-known result, cf. \cite[Proposition 9.17]{GuedjZeriahBook} or \cite[Proposition 2.2]{GuedjLuZeriahEnv}.
	\begin{lem}\label{lem_env_sing_vers}
		The lower semicontinuous regularization, $P[h^L]_*$, of $P[h^L]$ is a singular metric with a psh potential with minimal singularities.
		Moreover, if $L$ is a big line bundle and $h^L_i$ is a decreasing sequence of continuous metrics converging towards $h^L$ outside of a pluripolar subset, then $P[h^L_i]$ converges towards $P[h^L]$ outside of a pluripolar subset.
	\end{lem}
	\par 
	The following result, which will play a central role in our work, builds on contributions of many mathematicians, see \cite{BEGZ, WittNystrMonotonicity, LuChinhCompMA, LuNguen, DDLMonoton}, and is nowadays called the \textit{monotonicity formula}. 
	We use the version due to Vu \cite[Theorem 4.4]{VuDucRelNonpp}.
	\begin{thm}\label{thm_wn_monot}
		Let $T_j$, $T'_j$, $j = 1, \ldots, n$, be closed positive $(1, 1)$-currents on $X$ such that $T_j$, $T'_j$ are in the same cohomology class, and $T'_j$ is less singular than $T_j$ for every $j$. 
		Then, we have
		\begin{equation}
			\int_X T_1 \wedge \cdots \wedge T_n \leq \int_X T'_1 \wedge \cdots \wedge T'_n.
		\end{equation}
	\end{thm}
	\begin{rem}\label{rem_sam_sing_type}
		In particular, if the currents $T_i$ and $T'_i$ have the same singularity type, we have
		\begin{equation}
			\int_X T_1 \wedge \cdots \wedge T_n = \int_X T'_1 \wedge \cdots \wedge T'_n.
		\end{equation}
	\end{rem}
	We shall also rely on the fact that non-pluripolar products behave well with respect to monotone sequences of potentials.
	To state precisely the result that we shall use, recall that a function $u$ is called \textit{quasi-continuous} if for each $\varepsilon > 0$, there exists an open set $U$ such that $\operatorname{Cap}_\omega(U) < \varepsilon$ and the restriction of $u$ on $X \setminus U$ is continuous.
	Above $\operatorname{Cap}_\omega$ is the Monge-Ampère capacity, see \cite[Definition 4.16]{GuedjZeriahBook}.
	Most importantly for us, plurisubharmonic functions (and hence their differences) are known to be quasi-continuous, see \cite[Theorem 4.20]{GuedjZeriahBook}.
	\par 
	Several versions of the following two results are available, cf. \cite[Theorem 2.17]{BEGZ}. 
	Here we need two variations following easily from Darvas-Di Nezza-Lu \cite[Theorem 2.3 and Remark 2.5]{DDLMonoton} and \cite[Theorem 2.6 and Remark 3.4]{DDLSurvey}.
	\begin{thm}\label{thm_cont_begz}
		Let $\alpha_i$, $i \in \{1, \ldots, n\}$ be smooth closed real $(1, 1)$-forms on $X$ representing psef classes. 
		Suppose that $u_i, u_i^k \in \psh(X, \alpha_i)$ are such that $u_i^k$ increase towards $u_j$ outside of a pluripolar subset, as $k \to \infty$.
		Let $\chi_k, \chi : X \to \real$ be quasi-continuous and uniformly bounded such that $\chi_k$ converges to $\chi$, as $k \to \infty$, in capacity.
		We denote $T_i^k := \alpha_i + \ddc u_i^k$ and $T_i := \alpha_i + \ddc u_i$.
		Then $\chi_k \cdot T_1^k \wedge \cdots \wedge T_n^k$ converge weakly towards $\chi \cdot T_1 \wedge \cdots \wedge T_n$, as $k \to \infty$.
	\end{thm}
	\begin{proof}
		The authors of \cite{DDLSurvey} proved the analogue of Theorem \ref{thm_cont_begz} where they assumed that $u_i^k$ converge towards $u_j$ in capacity, the classes $\alpha_i$ are big, and the following inequality is satisfied 
		\begin{equation}\label{eq_increasing_mass_cond}
			\limsup_{k \to \infty} \int T_1^k \wedge \cdots \wedge T_n^k
			\leq
			\int T_1 \wedge \cdots \wedge T_n,
		\end{equation}
		\par 
		Note that monotone convergence is stronger than convergence in capacity, cf. \cite[Theorem 4.25]{GuedjZeriahBook}.
		Also, under our increasing assumption on $u_i^k$, $k \in \nat$, the inequality (\ref{eq_increasing_mass_cond}) follows immediately from Theorem \ref{thm_wn_monot}.
		This establishes Theorem \ref{thm_cont_begz} for big classes. 
		But in realms of Theorem \ref{thm_cont_begz}, it is sufficient for $\alpha_i$ to be psef, as for any $\epsilon > 0$, one can apply the result for $\alpha_i + \epsilon \omega$ for some K\"ahler form $\omega$, take a limit $\epsilon \to 0+$, and apply the multilinearity of the non-pluripolar product, cf. \cite[Proposition 1.4c)]{BEGZ}.
	\end{proof}
	Another version concerns the convergence result in the big setting.
	\begin{thm}\label{thm_cont_begz2}
		The conclusion of Theorem \ref{thm_cont_begz} remain valid if we assume that the cohomology classes of $\alpha_i$ are big, $u_i$ have minimal singularities, $u_i^k$ are monotone in $k$, and converge towards $u_j$, as $k \to \infty$, outside of a pluripolar subset.
	\end{thm}
	\begin{proof}
		The proof follows from \cite{DDLSurvey} exactly as in Theorem \ref{thm_cont_begz}, with the only additional observation being that the condition (\ref{eq_increasing_mass_cond}) is ensured by Theorem \ref{thm_wn_monot} and our assumption that $u_i$ have minimal singularities.
	\end{proof}
	\par 
	Finally, we state a version of the projection formula. 
	Although the formulation given below is likely to be far from optimal in terms of its hypotheses, it is the version that we shall use.
	Let $\pi : X \to Y$ be a surjective map between the compact complex manifolds of dimension $n$ and $\kappa$ respectively. 
	Assume $\alpha$ is a closed smooth $(1, 1)$-form on $Y$, and $\omega_0$ is a smooth semipositive closed $(1, 1)$-form from a big class on $X$ (for example, a K\"ahler metric).
	We denote by $\pi_* [\omega_0]^{n - \kappa}$ the integral of $\omega_0^{n - \kappa}$ along a general fiber of $\pi$.
	\begin{prop}\label{prop_proj_fla}
		Assume the cohomological class of $\alpha$ is big and that $\phi \in \psh(Y, \alpha)$ is bounded.
		Let $\chi$ be a quasi-continuous bounded function on $Y$.
		Then we have
		\begin{equation}
			\int_X \pi^* \chi \cdot (\pi^* \alpha + dd^c \pi^* \phi)^{\kappa} \wedge \omega_0^{n - \kappa}
			=
			\pi_*[\omega_0]^{n - \kappa}
			\cdot
			\int_Y \chi \cdot (\alpha + dd^c \phi)^{\kappa}.
		\end{equation}
	\end{prop}
	\begin{proof}
		Note that when $\phi$ is smooth, the result follows immediately from a combination of Sard's theorem and Ehresmann's lemma. 
		Let us now establish the result in general.
		\par 
		Let $\omega_Y$ (resp. $\omega$) be a K\"ahler form on $Y$ (resp. $X$).
		We fix $\epsilon > 0$; then by Demailly regularization theorem, cf. \cite{DemRegul}, there is a decreasing sequence of smooth $\phi_i : Y \to \real$, $\phi_i \in \psh(Y, \alpha + \epsilon \omega_Y)$, converging pointwise towards $\phi$.
		Then by Theorem \ref{thm_cont_begz2}, since $\phi$ is bounded (and hence has minimal singularities), we have
		\begin{equation}\label{eq_proj_2}
		\begin{aligned}
			&
			\int_Y \chi \cdot (\alpha + \epsilon \omega_Y + dd^c \phi)^{\kappa}
			=
			\lim_{i \to \infty} 
			\int_Y \chi \cdot (\alpha + \epsilon \omega_Y + dd^c \phi_i)^{\kappa},
			\\
			& \int_X \pi^* \chi \cdot (\pi^* (\alpha + \epsilon \omega_Y) + \epsilon \omega + dd^c \pi^* \phi)^{\kappa} \wedge \omega_0^{n - \kappa}
			\\
			&
			\qquad \qquad \qquad \qquad \qquad \qquad 
			=
			\lim_{i \to \infty} 
			\int_X \pi^* \chi \cdot (\pi^* (\alpha + \epsilon \omega_Y) + \epsilon \omega + dd^c \pi^* \phi_i)^{\kappa} \wedge \omega_0^{n - \kappa}.
		\end{aligned}
		\end{equation}
		By the multilinearity of the non-pluripolar product and the boundedness of $\chi$, there is $C > 0$, so that for any $\epsilon > 0$, $\psi \in \psh(Y, \alpha)$, $\psi_0 \in \psh(Y, \alpha + \epsilon \omega_Y)$,  we have
		\begin{equation}\label{eq_proj_4}	
		\begin{aligned}
			&
			\Big|
			\int_Y \chi \cdot (\alpha + dd^c \phi)^{\kappa}
			-
			\int_Y \chi \cdot (\alpha + \epsilon \omega_Y + dd^c \phi)^{\kappa}
			\Big|
			\leq 
			C \epsilon,
			\\
			&
			\Big|
			\int_X \pi^* \chi \cdot (\pi^* (\alpha + \epsilon \omega_Y) + \epsilon \omega + dd^c \pi^* \psi)^{\kappa} \wedge \omega_0^{n - \kappa}
			\\
			&
			\qquad \qquad \qquad \qquad \qquad \qquad 
			-
			\int_X \pi^* \chi \cdot (\pi^* \alpha + dd^c \pi^* \psi)^{\kappa} \wedge \omega_0^{n - \kappa}
			\Big|
			\leq 
			C \epsilon,
			\\
			&
			\Big|
			\int_X \pi^* \chi \cdot (\pi^* (\alpha + \epsilon \omega_Y) + \epsilon \omega + dd^c \pi^* \psi_0)^{\kappa} \wedge \omega_0^{n - \kappa}
			\\
			&
			\qquad \qquad \qquad \qquad \qquad \qquad 
			-
			\int_X \pi^* \chi \cdot (\pi^* (\alpha + \epsilon \omega_Y) + dd^c \pi^* \psi_0)^{\kappa} \wedge \omega_0^{n - \kappa}
			\Big|
			\leq 
			C \epsilon.
		\end{aligned}
		\end{equation}
		By applying the aforementioned smooth version of Proposition \ref{prop_proj_fla}, we obtain 
		\begin{equation}\label{eq_proj_5}
			\int_X \pi^* \chi \cdot (\pi^* (\alpha + \epsilon \omega_Y) + dd^c \pi^* \phi_i)^{\kappa} \wedge \omega_0^{n - \kappa}
			=
			\pi_* [\omega_0]^{n - \kappa}
			\cdot
			\int_Y \chi \cdot (\alpha + \epsilon \omega_Y + dd^c \phi_i)^{\kappa}.
		\end{equation}
		A combination of (\ref{eq_proj_2}), (\ref{eq_proj_4}) and (\ref{eq_proj_5}), yields that there is $C > 0$, so that for any $\epsilon > 0$, we have
		\begin{equation}
			\Big|
			\int_X \pi^* \chi \cdot (\pi^* \alpha +  dd^c \pi^* \phi)^{\kappa} \wedge \omega_0^{n - \kappa}
			-
			\pi_* [\omega_0]^{n - \kappa}
			\cdot
			\int_Y \chi \cdot (\alpha + dd^c \phi)^{\kappa}
			\Big|
			\leq 
			C \epsilon.
		\end{equation}
		The proof is now finished, as $\epsilon$ can be chosen arbitrarily small.
	\end{proof}
	\begin{rem}\label{rem_zero_prod}
		a)
		Analogously, we can see that for any $i > \kappa$, we have
		\begin{equation}
			\int_X \pi^* \chi \cdot (\pi^* \alpha + dd^c \pi^* \phi)^{i} \wedge \omega_0^{n - i}
			=
			0.
		\end{equation}
		\par 
		b) The same proof yields that if the $(1, 1)$-forms $\alpha_i$ (resp. $\phi_i \in \psh(Y, \alpha_i)$), $i = 1, \ldots, \kappa$, verify the assumptions of $\alpha$ (resp. $\phi$) from Proposition \ref{prop_proj_fla}, then we have
		\begin{equation}
			\int_X \pi^* \chi \cdot \prod_{i = 1}^{\kappa} (\pi^* \alpha_i + dd^c \pi^* \phi_i) \wedge \omega_0^{n - \kappa}
			=
			\pi_*[\omega_0]^{n - \kappa}
			\cdot
			\int_Y \chi \cdot  \prod_{i = 1}^{\kappa} (\alpha_i + dd^c \phi_i).
		\end{equation}
	\end{rem}

	\subsection{Fujita-type approximation theorem and resolutions}\label{sect_fujita}
	
	The main goal of this section is to recall one of the principal techniques used throughout this paper, namely the Fujita-type approximation theorem.
	For birational linear series, this result was established by Lazarsfeld-Mustață \cite[Theorem D]{LazMus}.
	The extension to arbitrary linear series -- which is the version relevant for our purposes -- is due to Kaveh-Khovanskii \cite[Corollary 3.11]{KavehKhov}.
	\par 
	To state this result, we fix a linear series $W$ verifying (\ref{eq_lin_series_ndg}). 
	For any $m \in \nat$, we consider the subring $[\sym W_m] \subset W$ generated by $W_m$. 
	Let $[\sym^k W_m]$ be the $k$-th graded component of $[\sym W_m]$. 
	The approximation theorem then goes as follows.
	\begin{thm}\label{thm_sym_subalgebra}
		For any $\epsilon > 0$, there is $m_0 \in \nat$ such that for any $m \geq m_0$, there is $k_0 \in \nat$ such that for any $k \geq k_0$, we have
		\begin{equation}
			\frac{\dim [\sym^k W_m]}{\dim W_{km}}
			\geq 
			1 - \epsilon.
		\end{equation}
	\end{thm}
	\par 
	We will also rely on the well-known observation, cf. \cite[Lemma 3.2]{ChangJow}, that the ring $[\sym W_m]$ coincides in high degrees with the section ring of an ample line bundle over a singular space. 
	To recall this, we follow the notations introduced in (\ref{eq_ki_maps}) for the singular space $Y_m$ and the rational map $\varphi_m$.
	We simultaneously resolve the indeterminacies of the Kodaira map, resolve the base ideal of $W_m$, and the singularities of $Y_m$, as shown in the diagram below
	\begin{equation}\label{eq_blow_up_kod}
		\begin{tikzcd}
		\widehat{X}_m \arrow[d, "\pi_m"] \arrow[dr, "f_m"] \arrow[r, "\widehat{\varphi}_m"]  & \widehat{Y}_m \arrow[d, "p_m"] & \\
X \arrow[r, dashed, "\varphi_m"] & Y_m \arrow[r, hook, "\iota_m"] & \mathbb{P}(W_m^*).
		\end{tikzcd}
	\end{equation}
	We denote by $A_m$ the restriction of the hyperplane bundle $\mathscr{O}(1)$ over $\mathbb{P}(W_m^*)$ to $Y_m$, and let $\widehat{L} := \pi_m^* L$.
	The resolution of the base ideal of $W_m$ yields an effective divisor $E_m$ on $\widehat{X}_m$, and the evaluation maps ${\rm{ev}}_x : W_m \to L_x^{\otimes m}$, $x \in X$, yield an isomorphism
	\begin{equation}\label{eq_iso_line_resol}
		\widehat{L}^{\otimes m} \simeq f_m^* A_m \otimes \mathscr{O}(E_m).
	\end{equation}
	\par 
	\begin{sloppypar}
	Let us now recall that for sufficiently large $k \in \nat$, there is an isomorphism 
	\begin{equation}\label{eq_sym_yk_rel}
		[\sym^k W_m]
		\simeq
		H^0(Y_m, A_m^{\otimes k}).
	\end{equation}
	First note that the restriction morphism from $\mathbb{P}(H^0(X, W_m)^*)$ to $Y_m$ induces a map 
	\begin{equation}\label{eq_restr_morph}
		H^0(\mathbb{P}(W_m^*), \mathscr{O}(k)) \to H^0(Y_m, A_m^{\otimes k}).
	\end{equation}
	Using the standard identification  $H^0(\mathbb{P}(W_m^*), \mathscr{O}(k)) \simeq \sym^k W_m$ and the definition of $Y_m$, this map factors through the evaluation map on $\sym^k W_m$, which gives precisely the map (\ref{eq_sym_yk_rel}). 
	The surjectivity of this map for $k$ sufficiently large follows from the surjectivity of the restriction morphism (\ref{eq_restr_morph}), implied by the Serre vanishing theorem.
	The injectivity follows from the fact that the sections $s_1, s_2 \in W_{km}$ coincide up to a constant when they have the same vanishing divisor on $X$, and the latter condition can be tested by comparing the evaluations on $s_1, s_2$, which geometrically corresponds to restricting $s_1, s_2$ along $Y_m$.
	\end{sloppypar}
	\par 
	Let us now collect a number of results which will explain the relation between various cohomology groups involved in (\ref{eq_blow_up_kod}).
	\par 
	By the Zariski main theorem, \cite[Corollary III.11.4]{HartsBook}, we have the isomorphism
	\begin{equation}\label{eq_isom_resol}
		\pi_m^* : H^0(X, L^{\otimes k}) \simeq H^0(\widehat{X}_m, \widehat{L}^{\otimes k}).
	\end{equation}
	We denote $\widehat{A}_m := p_m^* A_m$.
	Then we have natural inclusions
	\begin{equation}\label{eq_incl_resol_coh}
		H^0(Y_m, A_m^{\otimes k}) 
		\overset{p_m^*}{\hookrightarrow} 
		H^0(\widehat{Y}_m, \widehat{A}_m^{\otimes k}) 
		\overset{\widehat{\varphi}_m^*}{\hookrightarrow} 
		H^0(\widehat{X}_m, \widehat{L}^{\otimes k m}).
	\end{equation}
	While the first inclusion (\ref{eq_incl_resol_coh}) is not necessarily an isomorphism (if $Y_m$ is not normal), by the ampleness of $A_m$, it has a small codimension.
	More precisely, by \cite[proof of Lemma 2.2.3]{LazarBookI}, for any $\epsilon > 0$, there is $k_0 \in \nat$ such that for any $k \geq k_0$,
	\begin{equation}\label{eq_h0ym_yhatm}
		\frac{\dim H^0(Y_m, A_m^{\otimes k})}{\dim H^0(\widehat{Y}_m, \widehat{A}_m^{\otimes k})}
		\geq 
		1 - \epsilon.
	\end{equation}
	Let us introduce the following coherent sheaf $\mathcal{J}_m := \pi_* \mathscr{O}(\widehat{Y}_m)$.
	By the projection formula, cf. \cite[p. 140]{LazarBookI}, we then have
	\begin{equation}\label{eq_isom_sheaf}
		H^0(\widehat{Y}_m, \widehat{A}_m^{\otimes k})
		\simeq
		H^0(Y_m, A_m^{\otimes k} \otimes \mathcal{J}_m).
	\end{equation}
	
	\subsection{Fubini-Study maps and intersection theory}\label{sect_fs_inter}
	
	The main goal of this section is to establish the analytic formulas for the Kodaira-Iitaka dimension and the multiplicity, i.e. to prove Theorems \ref{thm_ki_form} and \ref{thm_vol_form}.
	\par 
	To do so, we first give a geometric interpretation of the Fubini-Study (singular) metric associated with a norm.
	We fix a norm $N_m$ on $W_m$ and induce from it the continuous metric on $\mathscr{O}(1)$ over $\mathbb{P}(W_m^*)$.
	We denote by $h^A_m$ the induced continuous metric on $A_m$, and by $\widehat{h}^A_m$ the pull-back of it to $\widehat{A}_m$.
	It is immediate to verify, cf. \cite[(3.10)]{FinGQBig}, that for the singular metric $FS(N_m)$ defined in (\ref{eq_fs_norm}), under the isomorphism (\ref{eq_iso_line_resol}), we have
	\begin{equation}\label{eq_fs_n_k_a_k_rel}
		\pi_m^* FS(N_m)
		=
		\varphi_m^*
		\widehat{h}^A_m
		\cdot
		h^{E_m}_{{\rm{sing}}}.
	\end{equation}
	Moreover, it is classical that the metric $\widehat{h}^A_m$ has a psh potential, cf. \cite[Theorem 4.1]{KobFinNeg}.
	\par 
	We will also need the following general result concerning Fubini-Study singular metrics.
	Recall that a graded norm $N = (N_k)_{k = 0}^{\infty}$, $N_k := \| \cdot \|_k$, over a linear series $W = \oplus_{k = 0}^{\infty} W_k \subset R(X, L)$ is called \textit{submultiplicative} if for any $k, l \in \nat^*$, $f \in W_k$, $g \in W_l$, we have
	\begin{equation}\label{eq_subm_s_ring}
		\| f \cdot g \|_{k + l} \leq 
		\| f \|_k \cdot
		\| g \|_l.
	\end{equation}
	\begin{lem}\label{lem_fs_sm}
		The sequence of the Fubini-Study (singular) metrics $FS(N_k)$, $k \in \nat^*$, is submultiplicative for any submultiplicative graded norm $N = (N_k)_{k = 0}^{\infty}$, i.e. for any $k, l \in \nat$, we have $FS(N_{k + l}) \leq FS(N_k) \cdot FS(N_l)$.
	\end{lem}
	\par 
	Note also that the Fubini-Study operator is monotone in the sense that if $N_k^0$ and $N_k^1$ are two norms defined on $W_k \subset H^0(X, L^{\otimes k})$, verifying $N_k^0 \geq N_k^1$, then we have
	\begin{equation}\label{eq_fs_monoton}
		FS(N_k^0) \geq FS(N_k^1).
	\end{equation}
	\par 
	It is also immediate that the Fubini-Study operator is monotone with respect to taking subspaces.
	More precisely, if we have an embedding $\iota: W_k^0 \subset W_k$, then for any norm $N_k$ on $W_k$,
	\begin{equation}\label{eq_fs_monoton2}
		FS(N_k) \leq FS(\iota^* N_k).
	\end{equation}
	\par 
	Let us now discuss how the Fubini-Study singular metrics depend on some of the modifications of the linear series that we introduced in Section \ref{sect_exampl}.
	\par 
	Let $\pi : X \to Y$ be a holomorphic map between two complex projective manifolds and $L_Y$ be a line bundle on $Y$.
	As described in Section \ref{sect_exampl}, any linear series $W$ on $(Y, L_Y)$ induces a linear series $\pi^* W$ on $R(X, \pi^* L_Y)$ by the pull-back.
	For any norm $N_k$ on $W_k \subset H^0(Y, L_Y^{\otimes k})$, we denote by $\pi^* N_k$ the induced norm on $\pi^* W_k$, given by the quotient map construction (\ref{eq_defn_quot_norm}) and the surjective map $\pi^* : W_k \to \pi^* W_k$.
	Immediate verification shows the following relation between the respective Fubini-Study singular metrics
	\begin{equation}\label{eq_fs_pullback}
		FS(\pi^* N_k)
		=
		\pi^* FS(N_k).
	\end{equation}	 
	\par 
	\begin{sloppypar}
	Let us now consider an arbitrary effective divisor $D$ on $X$, and a vector subspace $W_k \subset H^0(X, L^{\otimes k})$.
	Consider the vector subspace $W_k^D \subset H^0(X, L^{\otimes k} \otimes \mathscr{O}(D))$ given by a multiplication of $W_k$ by the canonical holomorphic section of $\mathscr{O}(D)$.
	Fix a norm on $N_k^D$ on $W_k^D$, and denote by $N_k$ the associated norm on $W_k$.
	There is the following relation between the respective Fubini-Study singular metrics
	\begin{equation}\label{eq_fs_nkd}
		FS(N_k^D)
		=
		FS(N_k)
		\cdot
		h^{D}_{{\rm{sing}}}.
	\end{equation}
	\end{sloppypar}
	\par 
	\begin{proof}[Proof of Theorems \ref{thm_ki_form} and \ref{thm_vol_form}]
		First of all, by Lemma \ref{lem_fs_sm}, the sequence of singular metrics
		\begin{equation}
			h^L_{2^k}
			:=
			FS({\textrm{Ban}}_{2^k}^{\infty}[W](h^L))^{\frac{1}{2^k}}
		\end{equation}
		is non-increasing.
		From now on, for brevity, we denote $m := 2^k$.
		By Theorem \ref{thm_cont_begz}, for any $i \in \nat$, we obtain the following convergence
		\begin{equation}
			\int_X c_1(L, P[W, h^L])^{i} \wedge \omega^{n - i}
			=
			\lim_{m \to \infty} \int_X c_1(L, h^L_m)^{i} \wedge \omega^{n - i}.
		\end{equation}
		We denote by $\widehat{h}^A_m$ the metric on $\widehat{A}_{m}$ induced by ${\textrm{Ban}}_{m}^{\infty}[W](h^L)$ as described before (\ref{eq_fs_n_k_a_k_rel}).
		By (\ref{eq_fs_n_k_a_k_rel}), we have
		\begin{equation}
			m \cdot \pi_m^* c_1(L, h^L_m)
			=
			\widehat{\varphi}_{m}^* c_1(\widehat{A}_{m}, \widehat{h}^A_m)
			+
			[E_m],
		\end{equation}
		where $[E_m]$ is the current of integration along $E_m$.
		From this and the fact that the non-pluripolar product doesn't charge the analytic subsets (as $E_m$), we deduce 
		\begin{equation}\label{eq_form_pullback}
			m^i \cdot \pi_m^* c_1(L, h^L_m)^{i} \wedge \pi_m^* \omega^{n - i}
			=
			\widehat{\varphi}_m^* c_1(\widehat{A}_m, \widehat{h}^A_m)^{i} \wedge \pi_m^* \omega^{n - i}.
		\end{equation}
		Recall that by (\ref{eq_dim_ki_yk}), we have $\dim \widehat{Y}_m = \kappa(W)$, for $m$ large enough.
		From this, Remark \ref{rem_zero_prod} and (\ref{eq_form_pullback}), we deduce immediately that for $i > \kappa(W)$, we have
		\begin{equation}\label{eq_i_big_zero}
			\int_X c_1(L, h^L_m)^{i} \wedge \omega^{n - i}
			=
			0.
		\end{equation}
		\par 
		Moreover, by Proposition \ref{prop_proj_fla} and (\ref{eq_form_pullback}), we have
		\begin{equation}\label{eq_int_l_a_id}
			m^{\kappa(W)} \cdot \int_X c_1(L, h^L_m)^{\kappa(W)} \wedge \omega^{n - \kappa(W)}
			=
			\widehat{\varphi}_{*}[ \pi_m^* \omega]^{n - \kappa(W)}
			\cdot
			\int_{\widehat{Y}_m} c_1(\widehat{A}_m, \widehat{h}^A_m)^{\kappa(W)},
		\end{equation}
		where $\widehat{\varphi}_{*}[ \pi_m^* \omega]^{n - \kappa(W)}$ is defined analogously to the quantity appearing in (\ref{eq_vol_form}).
		\par 
		Since $p_m$ and $\pi_m$ are birational modifications, we deduce that for $m$ large enough
		\begin{equation}\label{eq_bir_fib_integg}
			\widehat{\varphi}_{*}[ \pi_m^* \omega]^{n - \kappa(W)}
			=
			\varphi_{*} [\omega]^{n - \kappa(W)}.
		\end{equation}
		\par 
		Also, since $\widehat{h}^A_m$ is a continuous metric with a psh potential on $\widehat{A}_m$, we have
		\begin{equation}\label{eq_y_m_int_a_vol}
			\int_{\widehat{Y}_m} c_1(\widehat{A}_m, \widehat{h}^A_m)^{\kappa(W)}
			=
			{\rm{vol}}(\widehat{A}_m),
		\end{equation}
		by a version of (\ref{eq_vol_bouck}).
		Note, however, that in the current semi-ample setting the corresponding result follows immediately from Theorem \ref{thm_wn_monot} and the Hilbert-Samuel theorem.
		\par 
		Also, by Theorem \ref{thm_sym_subalgebra}, (\ref{eq_sym_yk_rel}) and (\ref{eq_incl_resol_coh}), we deduce that for any $\epsilon > 0$, there is $m_0 \in \nat$, so that for any $m \geq m_0$, we have
		\begin{equation}\label{eq_bol_w_a_est}
			(1 + \epsilon) 
			\cdot 
			m^{\kappa(W)}
			\cdot
			{\rm{vol}}_{\kappa}(W)
			\geq
			{\rm{vol}}(\widehat{A}_m)
			\geq
			(1 - \epsilon) 
			\cdot 
			m^{\kappa(W)}
			\cdot
			{\rm{vol}}_{\kappa}(W).
		\end{equation}
		By combining (\ref{eq_form_pullback}), (\ref{eq_int_l_a_id}), (\ref{eq_bir_fib_integg}), (\ref{eq_y_m_int_a_vol}) and (\ref{eq_bol_w_a_est}), we establish immediately Theorem \ref{thm_vol_form}.
		\par 
		Since ${\rm{vol}}_{\kappa}(W) > 0$, by Theorem \ref{thm_vol_form}, we deduce that
		\begin{equation}
			\int_X c_1(L, P[W, h^L])^{\kappa(W)} \wedge \omega^{n - \kappa(W)} > 0.
		\end{equation}
		This along with (\ref{eq_i_big_zero}) finishes the proof of Theorem \ref{thm_ki_form}.
	\end{proof}
	
	We conclude this section with the following remark. 
	At first glance, the condition (\ref{eq_kappa_w}) appears to depend on the choice of the Kähler form $\omega$ and the metric $h^L$. 
	In fact, this is not the case. 
	While this follows from Theorem \ref{thm_ki_form}, it can also be verified directly.
	\par 
	The independence with respect to $\omega$ follows from the observation that any two Kähler metrics bound each other from above and below up to a positive multiplicative constant.
	\par
	For the independence with respect to $h^L$, observe that the measure in (\ref{eq_meas_x_wedge}) is nonnegative, so its vanishing is equivalent to the vanishing of its integral. 
	The monotonicity formula, see Theorem \ref{thm_wn_monot}, then shows that this integral depends only on the singularity type of $P[W, h^L]$. 
	Since the singularity type of $P[W, h^L]$ is independent of the choice of $h^L$, the claim follows.

	\section{Comparison of the volumes of unit balls}\label{sect_vol_balls_compar}
	The main goal of this section is to relate the ratio of volumes of unit balls associated with sup-norms with the relative Monge-Ampère $\kappa$-energy.
	For this, in Section \ref{sect_alt_q_sch}, we provide an alternative quantization scheme for the envelope associated with the linear series.
	In Section \ref{sect_vol_balls}, building on this, the main result from Berman-Boucksom \cite{BermanBouckBalls} and Fujita type approximation theorem, Theorem \ref{thm_sym_subalgebra}, we establish Theorem \ref{thm_volumes_ratio}.
	
	\subsection{An alternative quantization scheme for the envelope}\label{sect_alt_q_sch}
	The main goal of this section is to describe how the sup-norms behave with respect to the identifications of the cohomology groups appearing in connection with the rational maps associated with the linear series, most notably (\ref{eq_sym_yk_rel}), (\ref{eq_isom_resol}) and (\ref{eq_incl_resol_coh}).
	Trough this, we present an alternative quantization scheme for the envelope (\ref{eq_env_w}).
	\par 
	We first begin with the description of the norms, as it will motivate the second part of the section. 
	Throughout the whole section, we use the notations introduced in (\ref{eq_blow_up_kod}).
	\par 
	We fix a continuous metric $h^L$ on $L$, a non-pluripolar subset $K \subset X$, and consider the induced sup-norm $\textrm{Ban}_k^{\infty}(K, h^L)$ on $H^0(X, L^{\otimes k})$.
	We denote by $\widehat{h}^L$ the pulled-back metric on $\widehat{L}$, and by $\widehat{K} := \pi_m^{-1}(K)$.
	Then under (\ref{eq_isom_resol}), we immediately have
	\begin{equation}
		\pi_m^* \textrm{Ban}_k^{\infty}(K, h^L)
		=
		\textrm{Ban}_k^{\infty}(\widehat{K}, \widehat{h}^L).
	\end{equation}
	\par 
	To describe the norm induced by (\ref{eq_incl_resol_coh}) on $H^0(\widehat{Y}_m, \widehat{A}_m^{\otimes k})$, we need to introduce a certain semimetric on $\widehat{A}_m$.
	For this, we note that the isomorphism (\ref{eq_iso_line_resol}) recasts as 
	\begin{equation}\label{eq_iso_line_resol2}
		\widehat{L}^{\otimes m} \simeq \widehat{\varphi}_m^* \widehat{A}_m \otimes \mathscr{O}(E_m).
	\end{equation}
	Using the isomorphism (\ref{eq_iso_line_resol2}), we define the semimetric $\widehat{\varphi}_{m, *} (\widehat{K}, \widehat{h}^L)$ on $\widehat{A}_m$ as follows: for any $y \in \widehat{Y}_m$, $l \in \widehat{A}_{m, y}$, we let 
	\begin{equation}\label{eq_max_metric}
		|l|_{\widehat{\varphi}_{m, *} (\widehat{K}, \widehat{h}^L)} 
		:=
		\begin{cases}
			\sup_{\substack{x \in \widehat{K} \\ \widehat{\varphi}_m(x) = y}} \big| s_{E_m}(x) \cdot \widehat{\varphi}_m^* l \big|_{(\widehat{h}^L_x)^k},
			&
			\text{ if } y \in \widehat{\varphi}_m(K),
			\\
			0, &
			\text{ otherwise.} 
		\end{cases}		
	\end{equation}
	where $s_{E_m}$ is the canonical section of the divisor $E_m$, defined as in the notation section.
	\par 
	Note that choosing a resolution different from $\widehat{Y}_m$ does not significantly alter the above semimetric: if one considers a resolution dominating the two given ones, then the pullbacks of the associated semimetrics agree.
	We need the following partial regularity result on $\widehat{\varphi}_{m, *} (\widehat{K}, \widehat{h}^L)$.
	\begin{lem}\label{lem_bnd_below_mettr}
		The semimetric $\widehat{\varphi}_{m, *} (\widehat{K}, \widehat{h}^L)$ is upper semicontinuous.
		Moreover, there is a non-pluripolar subset $E$ in $\widehat{Y}_m$, and a continuous metric $\widehat{h}^A_0$ on $\widehat{A}_m$ over $\widehat{Y}_m$ so that over $E$, we have $\widehat{\varphi}_{m, *} (\widehat{K}, \widehat{h}^L) \geq \widehat{h}^A_0$.
	\end{lem}
	\begin{proof}
		The upper semicontinuity follows immediately from the fact that $\widehat{\varphi}_{m}$ is a proper map and the semimetrics $\widehat{h}^L$ and $h^{E_m}_{{\rm{sing}}}$ are continuous.
		We now prove the second claim.
		Fix an arbitrary continuous metric $\widehat{h}^A_0$ on $\widehat{A}_m$ over $\widehat{Y}_m$, and consider the sequence of subsets $E^0_l \subset \widehat{Y}_m$, $l \in \nat^*$, defined by the condition that $\widehat{\varphi}_{m, *} (\widehat{K}, \widehat{h}^L) \geq \frac{1}{l} \widehat{h}^A_0$ over $E^0_l$.
		It suffices to prove that there exists $l \in \nat^*$ such that $E^0_l$ is non-pluripolar.
		Suppose, for the sake of contradiction, that this is not the case.
		Then the union $E := \cup E^0_l$ would be pluripolar, and consequently the set $\widehat{\varphi}_m^{-1}(E)$ would also be pluripolar.
		However, by (\ref{eq_max_metric}), we have $\widehat{K} \subset \widehat{\varphi}_m^{-1}(E) \cup E_m$, which leads to a contradiction, since $K$ (and hence $\widehat{K}$) is assumed to be non-pluripolar.
	\end{proof}
	\par 
	The reason why the above semimetric plays a crucial role in the current study is because immediately from the construction, we see that under the inclusion (\ref{eq_incl_resol_coh}), we have
	\begin{equation}\label{eq_kod_pullback_sup}
		 \widehat{\varphi}_m^* \textrm{Ban}_{km}^{\infty}(\widehat{K}, \widehat{h}^L)
		 =
		 \textrm{Ban}_{k}^{\infty}(\widehat{\varphi}_{m, *} (\widehat{K}, \widehat{h}^L)),
	\end{equation}
	where $\textrm{Ban}_{k}^{\infty}(\widehat{\varphi}_{m, *} (\widehat{K}, \widehat{h}^L))$ denotes the sup-norm on $H^0(\widehat{Y}_m, \widehat{A}_m^{\otimes k})$ induced by $\widehat{\varphi}_{m, *} (\widehat{K}, \widehat{h}^L)$.
	\par 
	We shall consider the envelope $P[\widehat{\varphi}_{m, *} (\widehat{K}, \widehat{h}^L)]$, which is a singular metric on $\widehat{A}_m$, defined as in (\ref{eq_sic_env}).
	By Lemmas \ref{lem_env_sing_vers} and \ref{lem_bnd_below_mettr}, we conclude that $P[\widehat{\varphi}_{m, *} (\widehat{K}, \widehat{h}^L)]_*$ has a psh potential.
	\par 
	Consider the following singular metric on $\hat{L}$:
	\begin{equation}\label{eq_hlm_defn}
		P[W_m, K, h^L] := \big( \widehat{\varphi}_m^* P[\widehat{\varphi}_{m, *} (\widehat{K}, \widehat{h}^L)]_* \cdot h^{E_m}_{\rm{sing}}\big)^{\frac{1}{m}}.
	\end{equation}
	Then $P[W_m, K, h^L]$ has a psh potential.
	Also, since $\pi_m$ is a birational modification, singular metrics with psh potentials on $\widehat{L}$ correspond to the pullbacks of singular metrics with psh potentials on $L$. 
	Hence $P[W_m, K, h^L]$ can be viewed as a singular metric on $L$, which we denote by the same symbol by a slight abuse of notation.
	\begin{lem}\label{lem_decr_quant_sch}
		The sequence of singular metrics $P[W_m, K, h^L]$ on $L$ decreases over multiplicative subsequences $m$, as $m = 2^k$, $k \in \nat$.
	\end{lem}
	\begin{proof}
		Note first that for any $k \in \nat^*$, we have the Veronese map
		\begin{equation}
			\mathbb{P}(W_m^*)
			\hookrightarrow
			\mathbb{P}({\rm{Sym}}^k W_m^*).
		\end{equation}
		On another hand, the multiplication yields the following rational map
		\begin{equation}
			\mathbb{P}(W_{km}^*)
			\dashrightarrow
			\mathbb{P}({\rm{Sym}}^k W_m^*).
		\end{equation}
		A combination of these two shows that we have the following commutative diagram
		\begin{equation}\label{eq_ym_ykm}
			\begin{tikzcd}
			& X \arrow[dl, dashed] \arrow[dr, dashed] & \\
	Y_m & & \arrow[ll, dashed] Y_{km}.
			\end{tikzcd}
		\end{equation}
		We consider a common resolution of singularities of the maps from (\ref{eq_ym_ykm}) that also resolves the singularities of the base loci of the associated linear series. 
		In particular, for $m$ and $m := km$, we obtain the diagram (\ref{eq_blow_up_kod}), together with the following diagram
		\begin{equation}\label{eq_commcomm00}
			\begin{tikzcd}
			& \widehat{X}_{km} \arrow[dl, "\widehat{\varphi}_m"'] \arrow[dr, "\widehat{\varphi}_{km}"] & \\
	\widehat{Y}_m & & \arrow[ll, "r_{k, m}"'] \widehat{Y}_{km}.
			\end{tikzcd}
		\end{equation}
		The base locus of ${\rm{Sym}}^k W_m$, which we embed in $\widehat{Y}_{km}$ through the pull-back $r_{k, m}^* H^0(\widehat{Y}_m, \widehat{A}_m^{\otimes k})$ and the isomorphism (\ref{eq_sym_yk_rel}), yields an effective divisor $D_{k, m}$ verifying
		\begin{equation}\label{eq_hakm_ham000}
			\widehat{A}_{km}
			\simeq
			r_{k, m}^* \widehat{A}_m^{\otimes k} \otimes \mathscr{O}(D_{k, m}).
		\end{equation}
		Note that the divisors are related in such a way that
		\begin{equation}
			\mathscr{O}(E_{m})^{\otimes k}
			\simeq
			\widehat{\varphi}_{km}^*
			\mathscr{O}(D_{k, m})
			\otimes
			\mathscr{O}(E_{km}).
		\end{equation}
		Immediately from the definitions, we have
		\begin{equation}\label{eq_hakm_ham}
			\widehat{\varphi}_{km, *} (\widehat{K}, \widehat{h}^L)
			\leq
			(r_{k, m}^* \widehat{\varphi}_{m, *} (\widehat{K}, \widehat{h}^L))^k \cdot h^{D_{k, m}}_{\rm{sing}}.
		\end{equation}
		Indeed, unwinding the definition (\ref{eq_max_metric}), using (\ref{eq_hakm_ham000}), we see that (\ref{eq_hakm_ham}) is equivalent to the following statement: for any $y' \in \widehat{Y}_{km}$, $y := r_{k, m}(y')$ and $l \in \widehat{A}_{m, y}$, we have 
		\begin{equation}\label{eq_max_metric_commm}
			\sup_{\substack{x \in \widehat{K} \\ \widehat{\varphi}_{km}(x) = y'}} \big| s_{E_{km}}(x) \cdot \widehat{\varphi}_{km}^* ( r_{k, m}^* l^{\otimes k} \cdot s_{D_{k, m}}(y')) \big|_{(\widehat{h}^L_x)^{km}}
			\leq
			\sup_{\substack{x \in \widehat{K} \\ \widehat{\varphi}_{m}(x) = y}} \big| s_{E_{m}}(x) \cdot \widehat{\varphi}_{m}^* l \big|_{(\widehat{h}^L_x)^{m}}^{k},
		\end{equation}
		which is immediate, as for any $x \in \widehat{K}$, verifying $\widehat{\varphi}_{km}(x) = y'$, we have $s_{D_{k, m}}(y') \otimes s_{E_{km}}(x) = s_{E_{m}}(x)^{\otimes k}$, and $x$ verifies $\widehat{\varphi}_{m}(x) = y$ by (\ref{eq_commcomm00}). 
		\par 
		Since the divisor $D_{k, m}$ is effective, the singular metric $h^{D_{k, m}}_{\rm{sing}}$ has a psh potential.
		Immediately from this and (\ref{eq_hakm_ham}), we obtain
		\begin{equation}
			P[\widehat{\varphi}_{km, *} (\widehat{K}, \widehat{h}^L)]
			\leq
			(r_{k, m}^* P[\widehat{\varphi}_{m, *} (\widehat{K}, \widehat{h}^L)])^k \cdot h^{D_{k, m}}_{\rm{sing}}.
		\end{equation}
		When recasted on the level of the metrics on the line bundle over $\widehat{X}_{km}$, it gives $P[W_{km}, K, h^L] \leq P[W_m, K, h^L]$, finishing the proof.
	\end{proof}
	\par 
	The main result of this section shows that the envelope $P[W, K, h^L]$, defined as in (\ref{eq_env_w}) and before (\ref{eq_diff_energy}), can be approximated by $P[W_m, K, h^L]$.
	\begin{thm}\label{prop_two_quant_sch}
		Assume $h^L$ is a continuous metric.
		Then, as $m \to \infty$, the following convergence of singular metrics on $L$ holds outside of a pluripolar subset
		\begin{equation}\label{eq_two_quant_sch}
			P[W_m, K, h^L]
			\to
			P[W, K, h^L].
		\end{equation}
	\end{thm}
	
	To establish Theorem \ref{prop_two_quant_sch}, the following result will be of crucial importance. 
	\begin{lem}\label{lem_sup_y_hat}
		Let $Y$ be an irreducible complex space and $A$ be a big line bundle over it.
		Consider a resolution of singularities $\pi: \widehat{Y} \to Y$ of $Y$, denote $\widehat{A} := \pi^* A$ and fix an upper semicontinuous semimetric $\widehat{h}^A$ on $\widehat{A}$.
		We assume that there is a non-pluripolar subset $E$ in $\widehat{Y}$, and a continuous metric $\widehat{h}^A_0$ on $\widehat{A}$ over $\widehat{Y}$ so that over $E$, we have $\widehat{h}^A \geq \widehat{h}^A_0$.
		We denote by ${\textrm{Ban}}_k^{\infty}[Y](\widehat{h}^A)$ the sup-norm on $H^0(Y, A^{\otimes k})$ induced by the embedding $H^0(Y, A^{\otimes k}) \hookrightarrow H^0(\widehat{Y}, \widehat{A}^{\otimes k})$ and the sup-norm ${\textrm{Ban}}_k^{\infty}(\widehat{h}^A)$ on $H^0(\widehat{Y}, \widehat{A}^{\otimes k})$ induced by $\widehat{h}^A$.
		Then, outside of the pluripolar subset, as $k \to \infty$, the following convergence holds
		\begin{equation}
			FS({\textrm{Ban}}_k^{\infty}[Y](\widehat{h}^A))^{\frac{1}{k}}
			\to
			P[\widehat{h}^A],
		\end{equation}
		where $P[\widehat{h}^A]$ is viewed as a singular metric on $A$ similarly to the discussion after (\ref{eq_hlm_defn}).
	\end{lem}
	\begin{proof}
		First of all, immediately from the definitions, for any $k \in \nat^*$, we have
		\begin{equation}\label{eq_fs_reg_1110}
			FS({\textrm{Ban}}_k^{\infty}[Y](\widehat{h}^A))^{\frac{1}{k}}
			\geq 
			\widehat{h}^A.
		\end{equation}
		Indeed, unwinding the definitions, it essentially says the following: for any $y \in \widehat{Y}$, $l \in \widehat{A}^{\otimes k}_y$, 
		\begin{equation}
			\sup_{\substack{s \in H^0(Y, A^{\otimes k}) \\ s(y) = \pi^* l, x \in \widehat{Y}}} | s(x) |_{(\widehat{h}^A_x)^k}
			\geq
			| l |_{(\widehat{h}^A_y)^k}.
		\end{equation}
		\par 
		Since $FS({\textrm{Ban}}_k^{\infty}[Y](\widehat{h}^A))$ has a psh potential, we deduce from (\ref{eq_fs_reg_1110}) that
		\begin{equation}\label{eq_fs_reg_111}
			FS({\textrm{Ban}}_k^{\infty}[Y](\widehat{h}^A))^{\frac{1}{k}}
			\geq 
			P(\widehat{h}^A).
		\end{equation} 
		\par 
		Let us now establish the opposite bound.
		To do so, we first establish that for the sup-norm ${\textrm{Ban}}_k^{\infty}(\widehat{h}^A)$ on $H^0(\widehat{Y}, \widehat{A}^{\otimes k})$, outside of the pluripolar subset, as $k \to \infty$, we have
		\begin{equation}\label{eq_fs_reg_env}
			FS({\textrm{Ban}}_k^{\infty}(\widehat{h}^A))^{\frac{1}{k}}
			\to
			P(\widehat{h}^A).
		\end{equation}
		Assume first that $\widehat{h}^A$ is a continuous metric.
		Then since the line bundle $\widehat{A}$ is big, (\ref{eq_fs_reg_env}) follows immediately from \cite[Theorem 2.16]{FinGQBig}, cf. also \cite{BermanEnvProj} for the related for the $L^2$-norms.
		\par 
		To prove it in general, we consider a sequence of continuous metrics $\widehat{h}^A_i$ decreasing towards $\widehat{h}^A$, as $i \to \infty$ (such a sequence exists since $\widehat{h}^A$ is upper semicontinuous).
		 Then by the already established statement (\ref{eq_fs_reg_env}) for continuous metrics and the monotonicity of the Fubini-Study operator, (\ref{eq_fs_monoton}), we deduce that outside of a pluripolar subset, for any $i \in \nat$, we have
		 \begin{equation}
		 	P(\widehat{h}^A_i)
		 	\geq
		 	\limsup_{k \to \infty} FS({\textrm{Ban}}_k^{\infty}(\widehat{h}^A))^{\frac{1}{k}}
		 \end{equation}
		 On another hand, by (\ref{eq_fs_reg_111}), we also deduce that
		 \begin{equation}
		 	\liminf_{k \to \infty} FS({\textrm{Ban}}_k^{\infty}(\widehat{h}^A))^{\frac{1}{k}}
		 	\geq
		 	P(\widehat{h}^A).
		 \end{equation}
		 Note, however, that by Lemma \ref{lem_env_sing_vers}, outside of a pluripolar subset, $P(\widehat{h}^A_i)$ converges to $P(\widehat{h}^A)$, as $i \to \infty$, which finishes the proof of (\ref{eq_fs_reg_env}) in full generality.
		\par 
		We now come back to the discussion of the proof of the opposite bound to (\ref{eq_fs_reg_111}).
		Let us introduce the following coherent sheaf $\mathcal{J} := \pi_* \mathscr{O}(\widehat{Y})$.
		As in (\ref{eq_isom_sheaf}), we have 
		\begin{equation}\label{eq_coh_yhat_y_simeq}
			H^0(\widehat{Y}, \widehat{A}^{\otimes k}) \simeq H^0(Y, A^{\otimes k} \otimes \mathcal{J}).
		\end{equation}
		Let us consider $r \in \nat$ large enough so that there is a non-zero $s_r \in H^0(Y, A^{\otimes r} \otimes \mathcal{J})$.
		Such $r$ exists by the ampleness of $A$.
		We consider a sequence of embeddings
		\begin{equation}
			j_k: H^0(Y, A^{\otimes k})
			\to
			H^0(Y, A^{\otimes (k + r)} \otimes \mathcal{J}), \qquad s \mapsto s \cdot s_r,
		\end{equation}
		where we used the identification (\ref{eq_coh_yhat_y_simeq}).
		Then it is clear that for $C := \sup_{y \in \widehat{Y}} |s_r(y)|_{(\widehat{h}^A)^r}$, for any $k \in \nat^*$, we have
		\begin{equation}\label{eq_pull_back_jk_norm}
			j_k^* {\textrm{Ban}}_{k + r}^{\infty}[Y](\widehat{h}^A)
			\leq
			C 
			\cdot
			{\textrm{Ban}}_{k}^{\infty}(\widehat{h}^A).
		\end{equation}
		\par 
		Now, note that if we denote by $h^{s_r}_{\rm{sing}}$ the singular metric on the line bundle associated with the divisor $D_r$ of $s_r \in H^0(Y, A^{\otimes k})$, defined so that $|s_r(x)|_{h^{s_r}_{\rm{sing}}} = 1$ outside of the zero locus of $s_r$, then under the isomorphism $A^{\otimes (k + r)} \simeq A^{\otimes k} \otimes \mathscr{O}(D_r)$, by (\ref{eq_fs_n_k_a_k_rel}) and (\ref{eq_fs_monoton2}), we immediately have
		\begin{equation}\label{eq_fs_yhat_y_triv}
			FS({\textrm{Ban}}_{k + r}^{\infty}[Y](\widehat{h}^A))
			\leq
			FS(j_k^* {\textrm{Ban}}_{k + r}^{\infty}[Y](\widehat{h}^A))
			\cdot
			h^{s_r}_{\rm{sing}}.
		\end{equation}
		From (\ref{eq_pull_back_jk_norm}), however, we have the following
		\begin{equation}\label{eq_fs_yhat_y_triv2}
			FS(j_k^* {\textrm{Ban}}_{k + r}^{\infty}[Y](\widehat{h}^A))
			\leq
			C
			\cdot
			FS({\textrm{Ban}}_{k}^{\infty}(\widehat{h}^A)).
		\end{equation}
		A combination of (\ref{eq_fs_reg_env}), (\ref{eq_fs_yhat_y_triv}) and (\ref{eq_fs_yhat_y_triv2}) yields immediately the result, as outside of an analytic subset, the following convergence $\lim_{k \to \infty} (h^{s_r}_{\rm{sing}} / P(\widehat{h}^A))^{\frac{1}{k}} = 1$ holds. 
	\end{proof}
	
	\begin{proof}[Proof of Theorem \ref{prop_two_quant_sch}]
		Let us associate for the norm ${\textrm{Ban}}_m^{\infty}[W](K, h^L)$ on $W_m$ the continuous metric $\widehat{h}^A_m$ on $\widehat{A}_m$ as before (\ref{eq_fs_n_k_a_k_rel}).
		We claim that the following inequality holds
		\begin{equation}\label{eq_triv_bnd_sup1}
			\widehat{\varphi}_{m, *} (\widehat{K}, \widehat{h}^L) \leq \widehat{h}^A_m.
		\end{equation}
		Indeed, unwinding the definitions, it essentially says the following immediate inequality: for any $y \in \widehat{Y}_m$, $l \in \widehat{A}_m$, we have
		\begin{equation}
			\sup_{x \in \widehat{\varphi}_m^{-1}(y) \cap \widehat{K}} \big| s_{E_m}(x) \cdot \widehat{\varphi}_m^* l \big|_{(\widehat{h}^L_x)^k}
			\leq
			\sup_{\substack{s \in H^0(Y_m, A_m) \\ s(p_m(y)) = l \\ x \in \widehat{K}}} \big| s_{E_m}(x) \cdot (f_m^* s)(x) \big|_{(\widehat{h}^L_x)^k}.
		\end{equation}
		\par
		As $\widehat{h}^A_m$ has a psh potential as described after (\ref{eq_fs_n_k_a_k_rel}), we deduce from (\ref{eq_triv_bnd_sup1}) that
		\begin{equation}
			P[\widehat{\varphi}_{m, *} (\widehat{K}, \widehat{h}^L)] \leq \widehat{h}^A_m.
		\end{equation}
		From this and (\ref{eq_fs_n_k_a_k_rel}), we deduce
		\begin{equation}
			\widehat{\varphi}_m^* P[\widehat{\varphi}_{m, *} (\widehat{K}, \widehat{h}^L)] \cdot h^{E_m}_{\rm{sing}}
			\leq
			\pi_m^* FS({\textrm{Ban}}_m^{\infty}[W](K, h^L)).
		\end{equation}
		From this and (\ref{eq_env_w}), by taking a limit $m \to \infty$, we see that outside of a pluripolar subset, we have
		\begin{equation}\label{eq_limsup_bnd}
			\limsup_{m \to \infty}
			P[W_m, K, h^L]
			\leq
			P[W, K, h^L],
		\end{equation}
		establishing one side of (\ref{eq_two_quant_sch}).
		\par 
		Let us now establish the opposite inequality. 
		We denote by $\textrm{Ban}_{k}^{\infty}[Y_m](\widehat{\varphi}_{m, *} (\widehat{K}, \widehat{h}^L))$ the norm on $H^0(Y_m, A_m^{\otimes k})$ induced by the sup-norm on $\widehat{Y}_m$ associated with $\widehat{\varphi}_{m, *} (\widehat{K}, \widehat{h}^L)$ and the first embedding from (\ref{eq_incl_resol_coh}).
		Immediately from (\ref{eq_incl_resol_coh}), (\ref{eq_fs_n_k_a_k_rel}), (\ref{eq_fs_monoton2}), (\ref{eq_kod_pullback_sup}), and the obvious fact $[{\rm{Sym}}^k W_m] \subset W_{km}$, we obtain the following
		\begin{equation}\label{eq_fs_abov_bnd}
			FS({\textrm{Ban}}_{km}^{\infty}[W](K, h^L))
			\leq
			f_m^*
			FS(\textrm{Ban}_{k}^{\infty}[Y_m](\widehat{\varphi}_{m, *} (\widehat{K}, \widehat{h}^L)))
			\cdot
			h^{E_m}_{\rm{sing}}.
		\end{equation}
		We take $k \to \infty$, apply Lemma \ref{lem_sup_y_hat} (which applies to $\widehat{\varphi}_{m, *} (\widehat{K}, \widehat{h}^L)$ because of Lemma \ref{lem_bnd_below_mettr}) to establish that outside of a pluripolar subset, we have
		\begin{equation}
			P[W, K, h^L]
			\leq
			P[W_m, K, h^L].
		\end{equation}
		We take a limit, as $m \to \infty$, to finally obtain that outside of a pluripolar subset, we have
		\begin{equation}\label{eq_liminf_bnd}
			P[W, K, h^L]
			\leq
			\liminf_{m \to \infty}
			P[W_m, K, h^L].
		\end{equation}
		A combination of (\ref{eq_limsup_bnd}) and (\ref{eq_liminf_bnd}) yields Theorem \ref{prop_two_quant_sch}.
	\end{proof}
		
	\subsection{Ratio of volumes and the relative Monge-Ampère $\kappa$-energy}\label{sect_vol_balls}
	
	The main goal of this section is to make a relation between the ratio of volumes of sup-norms and the relative Monge-Ampère $\kappa$-energy, i.e. to establish Theorem \ref{thm_volumes_ratio}.
	Our proof will build on the Fujita type approximation theorem, i.e. Theorem \ref{thm_sym_subalgebra}, on the result of Berman-Boucksom \cite{BermanBouckBalls} which establishes the corresponding statement for the complete linear series on big line bundles, and on the results of Section \ref{sect_alt_q_sch}, from which we borrow the notations in this section.
	\par 
	We fix a complex projective manifold $\widehat{Y}$ of dimension $\kappa \in \nat$, and a big line bundle $\widehat{A}$ over it.
	For any two continuous metrics $\widehat{h}^A_0$, $\widehat{h}^A_1$, on $\widehat{A}$, we define the associated envelopes $P[\widehat{h}^A_0]$, $P[\widehat{h}^A_1]$ as in (\ref{eq_sic_env}), and define the relative Monge-Ampère energy as follows
	\begin{multline}\label{eq_diff_energy_a}
		\mathscr{E}(P[\widehat{h}^A_0]_*) - \mathscr{E}(P[\widehat{h}^A_1]_*)
		:=
		\frac{1}{2(\kappa + 1)}
		\sum_{i = 0}^{\kappa} \int_{\widehat{Y}} \log\Big( \frac{P[\widehat{h}^A_1]_*}{P[\widehat{h}^A_0]_*} \Big)
		\cdot
		\\
		 \cdot c_1(\widehat{A}, P[\widehat{h}^A_0]_*)^i \wedge c_1(\widehat{A}, P[\widehat{h}^A_1]_*)^{\kappa - i}.
	\end{multline}
	Note that by the continuity of $\widehat{h}^A_i$, $i = 0, 1$, we have $P[\widehat{h}^A_i]_* = P[\widehat{h}^A_i]$.
	\par 
	Let us recall the following result due to Berman-Boucksom \cite{BermanBouckBalls}.
	For every $k \in \nat$, we fix a Hermitian norm $H_k$ on $H^0(\widehat{Y}, \widehat{A}^{\otimes k})$, which allows us to calculate the volumes $\vol(\cdot)$ of measurable subsets in $H^0(\widehat{Y}, \widehat{A}^{\otimes k})$. 
	Note that while such volumes depend on the choice of $H_k$, their ratio does not.
	For $i = 0, 1$, we denote by $\mathbb{B}_k[\widehat{h}^A_i]$ the unit balls in $H^0(\widehat{Y}, \widehat{A}^{\otimes k})$ corresponding to the sup-norms ${\textrm{Ban}}_k^{\infty}(\widehat{h}^A_i)$.
	\begin{thm}\label{thm_volumes_ratio_bb}
		The following formula for the asymptotic volume ratio holds
		\begin{equation}\label{rem_thm_isom2}
			\lim_{k \to \infty}
			\frac{1}{k^{\kappa + 1}} \log \Big( \frac{\vol(\mathbb{B}_k[\widehat{h}^A_0])}{\vol(\mathbb{B}_k[\widehat{h}^A_1])} \Big)
			=
			\mathscr{E}(P[\widehat{h}^A_0]) - \mathscr{E}(P[\widehat{h}^A_1]).
		\end{equation}
	\end{thm}
	\par 
	We also need a preliminary result which compares the volumes of balls on a finitely dimensional vector space $V$ an its subspace $E \subset V$.
	We denote $v := \dim V$, $e := \dim E$, and assume that for some $\epsilon > 0$, we have $e/v \geq 1 - \epsilon$.
	Let $N_0$, $N_1$ be two norms on $V$ and $N_0[E]$, $N_1[E]$ be their restrictions to $E \subset V$.
	We denote by $\mathbb{B}_0$, $\mathbb{B}_1$ (resp. $\mathbb{B}_0[E]$, $\mathbb{B}_1[E]$) the unit balls on $(V, N_0)$, $(V, N_1)$ (resp. $(E, N_0[E])$, $(E, N_1[E])$).
	We fix Hermitian norms $H_V$, $H_E$ on $V$ and $E$, which allow us to calculate the volumes $\vol(\cdot)$ of measurable subsets in $V$ and $E$. 
	\begin{lem}\label{lem_comp_rest}
		For any $C \geq 1$, verifying $N_0 \cdot \exp(-C) \leq N_1 \leq N_0 \cdot \exp(C)$, we have
		\begin{equation}\label{eq_comp_dist_two}
			\Big| \Big( \log \Big( \frac{\vol(\mathbb{B}_1)}{\vol(\mathbb{B}_0)} \Big) - \log \Big( \frac{\vol(\mathbb{B}_1[E])}{\vol(\mathbb{B}_0[E])} \Big) \Big|
			\leq
			3 \epsilon \cdot C \cdot v + 40 (1 + \log v) \cdot v.
		\end{equation}
	\end{lem}
	\begin{proof}
		It is an immediate consequence of \cite[Lemmas 2.22 and 3.3]{FinGQBig}.
	\end{proof}
	\par 
	By combining these two results, as the first step towards the proof of Theorem \ref{thm_sym_subalgebra}, we will establish a singular version of Theorem \ref{thm_volumes_ratio_bb}.
	To set up the stage, we fix an irreducible complex space $Y$ and an ample line bundle $A$ over it.
	Consider a resolution of singularities $\pi: \widehat{Y} \to Y$ of $Y$, denote $\widehat{A} := \pi^* A$ and fix two semimetrics $\widehat{h}^A_0$, $\widehat{h}^A_1$ on $\widehat{A}$ such that for some $C > 1$, we have 
	\begin{equation}\label{eq_metr_a_equising}
		C^{-1} \cdot \widehat{h}^A_0 \leq \widehat{h}^A_1 \leq C \cdot \widehat{h}^A_0,
	\end{equation}
	and such that $\widehat{h}^A_i$, $i = 0, 1$, are bounded from below over a non-pluripolar subset by a continuous metric.
	The latter condition in particular implies that both $P[\widehat{h}^A_0]_*$ and $P[\widehat{h}^A_1]_*$ have psh potentials by Lemma \ref{lem_env_sing_vers}.
	Along with (\ref{eq_metr_a_equising}), this implies that (\ref{eq_diff_energy_a}) is then well-defined for such $\widehat{h}^A_0$, $\widehat{h}^A_1$.
	\par 
	We denote by ${\textrm{Ban}}_k^{\infty}[Y](\widehat{h}^A_i)$, $i = 0, 1$, the sup-norms on $H^0(Y, A^{\otimes k})$ induced by the embedding $H^0(Y, A^{\otimes k}) \hookrightarrow H^0(\widehat{Y}, \widehat{A}^{\otimes k})$ and the sup-norm ${\textrm{Ban}}_k^{\infty}(\widehat{h}^A_i)$ on $H^0(\widehat{Y}, \widehat{A}^{\otimes k})$ induced by $\widehat{h}^A_i$.
	We denote by $\mathbb{B}_k[Y, \widehat{h}^A_i]$, $i = 0, 1$, the unit balls in $H^0(Y, A^{\otimes k})$ associated with these sup-norms.
	\begin{lem}\label{lem_bb_sing}
		For any upper semicontinuous semimetrics $\widehat{h}^A_0$, $\widehat{h}^A_1$ on $\widehat{A}$ verifying (\ref{eq_metr_a_equising}), and such that they are bounded from below over a non-pluripolar subset by a continuous metric, we have
		\begin{equation}\label{rem_thm_isom2}
			\lim_{k \to \infty}
			\frac{1}{k^{\kappa + 1}}
			\log\Big(
			\frac{\vol(\mathbb{B}_k[Y, \widehat{h}^A_0])}{\vol(\mathbb{B}_k[Y, \widehat{h}^A_1])}
			\Big)
			=
			\mathscr{E}(P[\widehat{h}^A_0]_*) - \mathscr{E}(P[\widehat{h}^A_1]_*).
		\end{equation}
	\end{lem}
	\begin{proof}
		Let us first reduce the problem to the comparison of balls on $H^0(\widehat{Y}, \widehat{A}^{\otimes k})$ instead of $H^0(Y, A^{\otimes k})$.
		Immediately from Lemma \ref{lem_comp_rest}, (\ref{eq_h0ym_yhatm}) and (\ref{eq_metr_a_equising}), following the notations of Theorem \ref{thm_volumes_ratio_bb}, we deduce that for any $\epsilon > 0$, there is $k_0 \in \nat$, such that for any $k \geq k_0$, we have
		\begin{equation}
			\Big|
			\log \Big( \frac{\vol(\mathbb{B}_k[Y, \widehat{h}^A_0])}{\vol(\mathbb{B}_k[Y, \widehat{h}^A_1])}
			\Big)
			-
			\log \Big( \frac{\vol(\mathbb{B}_k[\widehat{h}^A_0])}{\vol(\mathbb{B}_k[\widehat{h}^A_1])}
			\Big)
			\Big|
			\leq
			\epsilon k^{\kappa + 1}.
		\end{equation}
		From this, it suffices to show Theorem \ref{thm_volumes_ratio_bb} for $\widehat{h}^A_0$, $\widehat{h}^A_1$ as in Lemma \ref{lem_bb_sing}.
		\par 
		To see this, we fix an upper semicontinuous semimetric $\widehat{h}^A$ on $\widehat{A}$ which is bounded from below over a non-pluripolar subset by a continuous metric.
		Consider a sequence of continuous metrics $\widehat{h}^{A, 0}_i$ decreasing towards $\widehat{h}^A$, as $i \to \infty$ (such a sequence exists since $\widehat{h}^A$ is upper semicontinuous).
		Consider also an increasing sequence of continuous metrics $\widehat{h}^{A, 1}_i$, such that, as $i \to \infty$, $P[\widehat{h}^{A, 1}_i]$ increases towards $P[\widehat{h}^A]$ outside of a pluripolar subset.
		To see that such a sequence of metrics exists, consider the lower semicontinuous regularization $P[\widehat{h}^A]_*$.
		Our condition on the lower bound of $\widehat{h}^A$ implies by Lemma \ref{lem_env_sing_vers} that $P[\widehat{h}^A]_*$ has a psh potential. 
		Take $\widehat{h}^{A, 1}_i$ to be an increasing sequence of continuous metrics converging to $P[\widehat{h}^A]_*$ (such a sequence exists since $P[\widehat{h}^A]_*$ is lower semicontinuous).
		As $P[\widehat{h}^A] \geq P[\widehat{h}^{A, 1}_i] \geq \widehat{h}^{A, 1}_i$, and $P[\widehat{h}^A]_*$ coincides with $P[\widehat{h}^A]$ outside of a pluripolar subset, see \cite{BedfordTaylor}, the sequence $\widehat{h}^{A, 1}_i$ satisfies the above requirements.
		\par 
		Note that the tautological maximum principle, cf. \cite[Proposition 1.8]{BermanBouckBalls}, states that
		\begin{equation}
			{\textrm{Ban}}_k^{\infty}(\widehat{h}^A)
			=
			{\textrm{Ban}}_k^{\infty}(P[\widehat{h}^A]).
		\end{equation}
		This implies that the unit balls, $\mathbb{B}_k[\widehat{h}^{A, 0}_i]$, $\mathbb{B}_k[\widehat{h}^{A}]$, $\mathbb{B}_k[\widehat{h}^{A, 1}_i]$, associated with ${\textrm{Ban}}_k^{\infty}(\widehat{h}^{A, 0}_i)$, ${\textrm{Ban}}_k^{\infty}(\widehat{h}^A)$ and ${\textrm{Ban}}_k^{\infty}(\widehat{h}^{A, 1}_i)$, verify the following inclusions
		\begin{equation}\label{eq_e_cont_incr00}
			\mathbb{B}_k[\widehat{h}^{A, 0}_i] \subset \mathbb{B}_k[\widehat{h}^{A}] \subset \mathbb{B}_k[\widehat{h}^{A, 1}_i]
		\end{equation}
		\par
		Now, by applying Theorem \ref{thm_volumes_ratio_bb}, we deduce
		\begin{equation}\label{eq_e_cont_incr0}
			\lim_{k \to \infty}
			\frac{1}{k^{\kappa + 1}} \log \Big( \frac{\vol(\mathbb{B}_k[\widehat{h}^{A, 0}_i])}{\vol(\mathbb{B}_k[\widehat{h}^{A, 1}_i])}
			\Big)
			=
			\mathscr{E}(P[\widehat{h}^{A, 0}_i]) - \mathscr{E}(P[\widehat{h}^{A, 1}_i]).
		\end{equation}
		Note, however, that Theorem \ref{thm_cont_begz2} implies that
		\begin{equation}\label{eq_e_cont_incr}
			\lim_{i \to \infty}
			\mathscr{E}(P[\widehat{h}^{A, 0}_i]) - \mathscr{E}(P[\widehat{h}^{A, 1}_i])
			=
			0.
		\end{equation}
		Indeed, Theorem \ref{thm_cont_begz2} applies to each term in the sum (\ref{eq_diff_energy_a}) by the fact that $P[\widehat{h}^A_i]$ have potentials with minimal singularities. 
		By applying the above approximations for both semimetrics, $\widehat{h}^A_i$, $i = 0, 1$, we see by (\ref{eq_e_cont_incr00}), (\ref{eq_e_cont_incr0}) and (\ref{eq_e_cont_incr}) that Theorem \ref{thm_volumes_ratio_bb} for upper semicontinuous semimetrics reduces to Theorem \ref{thm_volumes_ratio_bb} for continuous metrics, which finishes the proof.
	\end{proof}
	The following technical corollary will be of great use later on.
	We conserve the notation introduced in Lemma \ref{lem_bb_sing}.
	\begin{lem}\label{lem_twist_rho}
		Let $\widehat{h}^A$ be an upper semicontinuous semimetric on $\widehat{A}$, which is bounded from below over a non-pluripolar subset by a continuous metric, and let $\nu : \widehat{Y} \to [0, 1]$ be a continuous function, so that $\nu^{-1}(0)$ is an analytic subvariety (or, more generally, a pluripolar subset).
		We denote by $\mathbb{B}_k[Y, \widehat{h}^A]$ (resp. $\mathbb{B}_k[Y, \nu, \widehat{h}^A]$) the unit ball on $H^0(Y, A^{\otimes k})$ associated with the sup-norm induced by $\widehat{h}^A$ (resp. the sup-norm induced by the semimetric $\nu \cdot (\widehat{h}^A)^{\otimes k}$).
		Then
		\begin{equation}\label{eq_twist_rho}
			\lim_{k \to \infty}
			\frac{1}{k^{\kappa + 1}}
			\log\Big(
			\frac{\vol(\mathbb{B}_k[Y, \widehat{h}^A])}{\vol(\mathbb{B}_k[Y, \nu, \widehat{h}^A])}
			\Big)
			=
			0.
		\end{equation}
	\end{lem}
	\begin{proof}
		Let us consider the subsets $E_i := \{ y \in \widehat{Y} : \nu(y) \geq 1/i \}$ and the sequence of semimetrics $\widehat{h}^A_i := 1_{E_i} \cdot \widehat{h}^A$, where $1_{E_i}$ is the indicator function.
		Then, as $i \to \infty$, $\widehat{h}^A_i$ increases towards $\widehat{h}^A$.
		By repeating the argument from the proof of Lemma \ref{lem_bnd_below_mettr}, we see also that there is $i_0 \in \nat$, such that for any $i \geq i_0$, the semimetric $\widehat{h}^A_i$ is bounded from below over a non-pluripolar subset by a continuous metric.
		Note also that $\widehat{h}^A_i$ is upper semicontinuous as $E_i$ are closed by the continuity of $\nu$.
		\par 
		For every $i \in \nat$, we obviously have
		\begin{equation}
			\mathbb{B}_k[Y, \widehat{h}^A]
			\subset
			\mathbb{B}_k[Y, \nu, \widehat{h}^A]
			\subset
			i \cdot \mathbb{B}_k[Y, \widehat{h}^A_i].
		\end{equation}
		This along with the fact that $\dim H^0(Y, A^{\otimes k}) \leq C k^{\kappa}$ for $k \in \nat$ large enough, implies the following
		\begin{equation}\label{rem_thm_isom300}
			\limsup_{k \to \infty}
			\frac{1}{k^{\kappa + 1}}
			\Big|\log
			\Big(
			\frac{\vol(\mathbb{B}_k[Y, \widehat{h}^A])}{\vol(\mathbb{B}_k[Y, \nu, \widehat{h}^A])}\Big) \Big|
			\leq
			\limsup_{k \to \infty}
			\frac{1}{k^{\kappa + 1}}
			\Big|
			\log\Big(
			\frac{\vol(\mathbb{B}_k[Y, \widehat{h}^A])}{\vol(\mathbb{B}_k[Y, \widehat{h}^A_i])}\Big)
			\Big|.
		\end{equation}
		\par 
		However, applying Lemma \ref{lem_bb_sing} yields 
		\begin{equation}\label{rem_thm_isom3}
			\lim_{k \to \infty}
			\frac{1}{k^{\kappa + 1}}
			\log\Big(
			\frac{\vol(\mathbb{B}_k[Y, \widehat{h}^A])}{\vol(\mathbb{B}_k[Y, \widehat{h}^A_i])}\Big)
			=
			\mathscr{E}(P[\widehat{h}^A]_*) - \mathscr{E}(P[\widehat{h}^A_i]_*).
		\end{equation}
		Moreover, since $\widehat{h}^A_i$ increase towards $\widehat{h}^A$, as $i \to \infty$, we deduce by Lemma \ref{lem_env_sing_vers} that $P[\widehat{h}^A_i]_*$ increase towards $P[\widehat{h}^A]_*$ outside of a pluripolar subset.
		Then, exactly as in (\ref{eq_e_cont_incr}), we deduce
		\begin{equation}\label{eq_e_cont_incr2}
			\lim_{i \to \infty}
			\mathscr{E}(P[\widehat{h}^A]_*) - \mathscr{E}(P[\widehat{h}^A_i]_*)
			=
			0.
		\end{equation}
		A combination of (\ref{rem_thm_isom300}), (\ref{rem_thm_isom3}) and (\ref{eq_e_cont_incr2}) easily yields the proof.
	\end{proof}
	\par 
	Let us now relate the relative Monge-Ampère energy, (\ref{eq_diff_energy}), with the relative Monge-Ampère $\kappa$-energy, (\ref{eq_diff_energy_a}).
	We fix two continuous metrics $h^L_0$ and $h^L_1$ on $L$, denote by $\widehat{h}^L_0$, $\widehat{h}^L_0$, their pullbacks to $\widehat{h}^L$, and by $\widehat{\varphi}_{m, *} (\widehat{K}, \widehat{h}^L_0)$, $\widehat{\varphi}_{m, *} (\widehat{K}, \widehat{h}^L_1)$ the semimetrics on $\widehat{A}_m$ induced by them as in (\ref{eq_max_metric}).
	\begin{lem}\label{lem_ma_kappa_ma_usual}
		The following relation between the relative Monge-Ampère energy and the relative Monge-Ampère $\kappa$-energy holds
		\begin{multline}\label{eq_lem_ma_kappa_ma_usual}
			\mathscr{E}_{\kappa}(P[W, K, h^L_0]) - \mathscr{E}_{\kappa}(P[W, K, h^L_1])
			\\
			=
			\lim_{m \to \infty}
			\frac{1}{m^{\kappa(W) + 1}}
			\cdot
			\Big(
			\mathscr{E}(P[\widehat{\varphi}_{m, *} (\widehat{K}, \widehat{h}^L_0)]_*) - \mathscr{E}(P[\widehat{\varphi}_{m, *} (\widehat{K}, \widehat{h}^L_1)]_*)
			\Big).
		\end{multline}
	\end{lem}
	\begin{rem}\label{rem_ma_kappa_ma_usual}
		a) Since the right-hand side of (\ref{eq_lem_ma_kappa_ma_usual}) is independent of $\omega$, the left-hand side must also be independent of $\omega$.
		\par 
		b) Note that the relative Monge-Ampère energy is monotonic and verifies the cocycle property, cf. \cite[Proposition 10.28]{GuedjZeriahBook}.
		By (\ref{eq_lem_ma_kappa_ma_usual}), the same continues to hold for the relative Monge-Ampère $\kappa$-energy.
		In particular, for any continuous metrics $h^L_i$, $i = 0, 1, 2$, we have
		\begin{multline}
			\mathscr{E}_{\kappa}(P[W, K, h^L_0]) - \mathscr{E}_{\kappa}(P[W, K, h^L_2])
			=
			\big(\mathscr{E}_{\kappa}(P[W, K, h^L_0]) - \mathscr{E}_{\kappa}(P[W, K, h^L_1])\big)
			\\
			+
			\big(\mathscr{E}_{\kappa}(P[W, K, h^L_1]) - \mathscr{E}_{\kappa}(P[W, K, h^L_2])\big),
		\end{multline}
		and so it is natural to regard the relative Monge-Ampère $\kappa$-energy as a difference.
		\par 
		While it won't be used in what follows, we nevertheless note that by the same reasons, the relative Monge-Ampère $\kappa$-energy is concave in the sense that the function 
		\begin{equation}
			t \mapsto \mathscr{E}_{\kappa}(P[W, K, h^L_0 \cdot \exp(tf) ]) - \mathscr{E}_{\kappa}(P[W, K, h^L_0]),
		\end{equation}
		is concave for any continuous $f$.
	\end{rem}
	\begin{proof}
		We denote by $P[W_m, K, h^L_i]$, $i = 0, 1$ the singular metrics on $L$ defined as in (\ref{eq_hlm_defn}) but for $h^L_i$ instead of $h^L$.
		Then immediately from the definitions and the fact that the nonpluripolar product doesn't put mass on analytic subsets, we have
		\begin{multline}
		\begin{aligned}
			&
			\widehat{\varphi}_m^* \Big( c_1(\widehat{A}_m, P[\widehat{\varphi}_{m, *} (\widehat{K}, \widehat{h}^L_0)]_*)^i \wedge c_1(\widehat{A}_m, P[\widehat{\varphi}_{m, *} (\widehat{K}, \widehat{h}^L_1)]_*)^{\kappa(W) - i} \Big) \wedge \omega^{n - \kappa(W)}
			\\
			&
			\qquad \qquad 
			=
			m^{\kappa(W)} \cdot c_1(L, P[W_m, K, h^L_0])^i \wedge c_1(L, P[W_m, K, h^L_1])^{\kappa(W) - i}  \wedge \omega^{n - \kappa(W)},
			\\
			&
			\widehat{\varphi}_m^*
			\log\Big( \frac{P[\widehat{\varphi}_{m, *} (\widehat{K}, \widehat{h}^L_1)]_*}{P[\widehat{\varphi}_{m, *} (\widehat{K}, \widehat{h}^L_0)]_*} \Big)
			=
			m
			\cdot
			\log\Big( \frac{P[W_m, K, h^L_1]}{P[W_m, K, h^L_0]} \Big).
		\end{aligned}
		\end{multline}
		From this, Proposition \ref{prop_proj_fla} and Remark \ref{rem_zero_prod}, we deduce
		\begin{multline}
			\frac{1}{m^{\kappa(W) + 1}}
			\cdot
			\Big(
			\mathscr{E}(P[\widehat{\varphi}_{m, *} (\widehat{K}, \widehat{h}^L_0)]_*) - \mathscr{E}(P[\widehat{\varphi}_{m, *} (\widehat{K}, \widehat{h}^L_1)]_*)
			\Big)
			\\
			=
			\frac{1}{2(\kappa(W) + 1) \cdot \varphi_{*} [\omega]^{n - \kappa(W)}}
			\sum_{i = 0}^{\kappa(W)} \int_{X} \log\Big( \frac{P[W_m, K, h^L_1]}{P[W_m, K, h^L_0]} \Big)
			\cdot
			\\
			 \cdot c_1(L, P[W_m, K, h^L_0])^i \wedge c_1(L, P[W_m, K, h^L_1])^{\kappa(W) - i}  \wedge \omega^{n - \kappa(W)}.
		\end{multline}
		Now, remark that by (\ref{eq_metr_a_equising}), there is $C > 1$ such that for any $m \in \nat$, we have
		\begin{equation}
			C^{-1}
			\le
			\frac{P[W_m, K, h^L_1]}{P[W_m, K, h^L_0]}
			\leq
			C.
		\end{equation}
		Note also that since the singular metrics $P[W_m, K, h^L_i]$, $i = 0, 1$ have psh potentials, the function $\log( \frac{P[W_m, K, h^L_1]}{P[W_m, K, h^L_0]})$ is quasi-continuous by the discussion before Theorem \ref{thm_cont_begz}.
		From this, Lemmas \ref{lem_bnd_below_mettr}, \ref{lem_decr_quant_sch} and Theorem \ref{prop_two_quant_sch}, we can apply Theorem \ref{thm_cont_begz}, which yields Lemma \ref{lem_ma_kappa_ma_usual}.
	\end{proof}
	
	\begin{proof}[Proof of Theorem \ref{thm_volumes_ratio}]
		First of all, for any $m \in \nat^*$, we denote by $\mathbb{B}_k[[{\rm{Sym}} W_m], K, h^L_i]$, $i = 0, 1$, the restriction of the unit ball $\mathbb{B}_{km}[W, K, h^L_0]$ on the subspace $[{\rm{Sym}}^k W_m] \subset W_{km}$.
		Then by Theorem \ref{thm_sym_subalgebra} and Lemma \ref{lem_comp_rest}, we deduce that for any $\epsilon > 0$, there are $m \in \nat^*$, $k_0 \in \nat$, such that for any $k \geq k_0$, we have
		\begin{equation}\label{eq_thm_volumes_ratio_0}
			\Big|
			\log\Big(
			\frac{\vol(\mathbb{B}_{km}[W, K, h^L_0])}{\vol(\mathbb{B}_{km}[W, K, h^L_1])}
			\Big)
			-
			\log\Big(
			\frac{\vol(\mathbb{B}_k[[{\rm{Sym}} W_m], K, h^L_0])}{\vol(\mathbb{B}_k[[{\rm{Sym}} W_m], K, h^L_1])}
			\Big)
			\Big|
			\leq
			\epsilon k^{\kappa + 1}.
		\end{equation}
		Note, however, that taken into account the identifications (\ref{eq_sym_yk_rel}) and (\ref{eq_kod_pullback_sup}), we deduce that if we denote by $\mathbb{B}_k[Y_m, \widehat{\varphi}_{m, *} (\widehat{K}, \widehat{h}^L_i)]$, $i = 0, 1$, the unit balls in $H^0(Y, A^{\otimes k})$ associated with the pullbacks $\widehat{h}^L_i$ on $\widehat{L}$ of the metrics $h^L_i$ as in Lemma \ref{lem_bb_sing}, then
		\begin{equation}\label{eq_thm_volumes_ratio_1}
			\log\Big(\frac{\vol(\mathbb{B}_k[[{\rm{Sym}} W_m], K, h^L_0])}{\vol(\mathbb{B}_k[[{\rm{Sym}} W_m], K, h^L_1])}\Big)
			=
			\log\Big(
			\frac{\vol(\mathbb{B}_k[Y_m, \widehat{\varphi}_{m, *} (\widehat{K}, \widehat{h}^L_0)])}{\vol(\mathbb{B}_k[Y_m, \widehat{\varphi}_{m, *} (\widehat{K}, \widehat{h}^L_1)])}
			\Big).
		\end{equation}
		Combining Lemmas \ref{lem_bnd_below_mettr}, \ref{lem_bb_sing}, \ref{lem_ma_kappa_ma_usual}, (\ref{eq_thm_volumes_ratio_0}) and (\ref{eq_thm_volumes_ratio_1}), we see that for any $\epsilon > 0$, there are $m \in \nat^*$, $k_0 \in \nat$, such that for any $k \geq k_0$, we have
		\begin{equation}\label{eq_thm_volumes_ratio_1000}
			\Big|
			\frac{1}{(km)^{\kappa(W) + 1}}
			\log\Big(\frac{\vol(\mathbb{B}_{km}[W, K, h^L_0])}{\vol(\mathbb{B}_{km}[W, K, h^L_1])}
			\Big)
			-
			\big(
			\mathscr{E}_{\kappa}(P[W, K, h^L_0]) - \mathscr{E}_{\kappa}(P[W, K, h^L_1])
			\big)
			\Big|
			\leq
			\epsilon.
		\end{equation}
		Let us now establish that the above estimate continues to hold if $mk$ is replaced by $k$.
		\par 
		For this, we fix $r \in \nat$ so that there is a non-zero section $s \in W_r$.
		Define the following functions $\nu_i(x) := |s(x)|_{(h^L_i)^k}$, $i = 0, 1$.
		Consider the embedding 
		\begin{equation}
			j_k: [{\rm{Sym}}^k W_m] \to W_{mk + r}, \qquad s \mapsto s \cdot s_r.
		\end{equation}
		Note that, taken into account the identifications (\ref{eq_sym_yk_rel}) and (\ref{eq_kod_pullback_sup}), the restriction of the unit ball $\mathbb{B}_{km + r}[W, K, h^L_i]$ to the subspace $j_k([{\rm{Sym}}^k W_m])$, is given -- in the notations of Lemma \ref{lem_twist_rho} -- by $\mathbb{B}_k[Y_m, \nu_i, \widehat{\varphi}_{m, *} (\widehat{K}, \widehat{h}^L_0)]$.
		From this, Theorem \ref{thm_sym_subalgebra} and Lemmas \ref{lem_comp_rest}, \ref{lem_twist_rho}, we deduce the following version of (\ref{eq_thm_volumes_ratio_0}) and (\ref{eq_thm_volumes_ratio_1}): for any $\epsilon > 0$, there are $m \in \nat^*$, $k_0 \in \nat$, such that for any $k \geq k_0$,
		\begin{equation}\label{eq_thm_volumes_ratio_0nn}
			\Big|
			\log\Big(\frac{\vol(\mathbb{B}_{km + r}[W, K, h^L_0])}{\vol(\mathbb{B}_{km + r}[W, K, h^L_1])}
			\Big)
			-
			\log\Big(
			\frac{\vol(\mathbb{B}_k[Y_m, \widehat{\varphi}_{m, *} (\widehat{K}, \widehat{h}^L_0)])}{\vol(\mathbb{B}_k[Y_m, \widehat{\varphi}_{m, *} (\widehat{K}, \widehat{h}^L_1)])}
			\Big)
			\Big|
			\leq
			\epsilon k^{\kappa + 1}.
		\end{equation}
		By repeating the argument preceding (\ref{eq_thm_volumes_ratio_1000}), we obtain that (\ref{eq_thm_volumes_ratio_1000}) holds with $mk + r$ in place of $mk$.
		Since our linear series satisfies (\ref{eq_lin_series_ndg}), this applies to $r$ representing all residues modulo $m$. 
		Therefore, (\ref{eq_thm_volumes_ratio_1000}) holds for $k$ instead of $mk$, completing the proof.
 	\end{proof}

	\section{Asymptotics of the partial Bergman kernels}\label{sect_pbk}
	The main objective of this section is to investigate two closely related problems: the differentiability of the relative Monge-Ampère $\kappa$-energy and the asymptotic behavior of the partial Bergman kernels.
	More precisely, in Section \ref{sect_no_weak} we present an explicit example showing that partial Bergman kernels may diverge, and that the relative Monge-Ampère $\kappa$-energy may fail to be differentiable, even though both properties are known to hold in the case of complete linear series on big line bundles.
	Then, in Section \ref{sect_different}, we prove a partial differentiability result for the relative Monge-Ampère $\kappa$-energy and use it to deduce Theorem \ref{thm_part_bk_get}.
	
	\subsection{An example of divergent partial Bergman kernels}\label{sect_no_weak}
	The main goal of this section is to show that partial Bergman kernels do not converge in general, and that the relative Monge-Ampère $\kappa$-energy fails to be differentiable in general.
	This will be achieved by providing two closely related counterexamples.
	\par 
	\textbf{Example.} Divergence of the partial Bergman kernel.
	\par 
	We consider the projective space $X := \mathbb{P}^1$ and the hyperplane bundle $L = \mathscr{O}(1)$. 
	We view $\mathbb{C}$ as an affine chart in $\mathbb{P}^1 = \mathbb{C} \cup \{\infty\}$ with the standard holomorphic coordinate $z$, and consider a non-zero section $\sigma \in H^0(X, L)$ vanishing at $\infty$. 
	The division by $\sigma^{k}$ gives an isomorphism between $H^0(X, L^{\otimes k})$ and the space of polynomials of degree $\leq k$ in $z$. 
	We endow $L$ with a continuous metric $h^L$, such that 
	\begin{equation}
		|\sigma(x)|_{h^L} = \frac{|x| + 1}{2}, \text{ for } x \in \mathbb{D}.
	\end{equation}
	Define the envelope, $P[\mathbb{D}, h^L]$, as
	\begin{equation}
		P[\mathbb{D}, h^L] 
		:=
		\inf 
		\Big\{
			h^L_0 : h^L_0 \geq h^L \text{over $\mathbb{D}$ and } h^L_0 \text{ has a psh pential}
		\Big\}.
	\end{equation}
	Clearly, it verifies the following inequality
	\begin{equation}
		P[\mathbb{D}, h^L] 
		\geq
		P[\mathbb{S}^1, h^L],
	\end{equation}
	where $P[\mathbb{S}^1, h^L]$ is defined analogously to $P[\mathbb{D}, h^L]$.
	A classical calculation, cf. \cite[Exercise 4.13]{GuedjZeriahBook}, shows that the metric $h^L_0 := P[\mathbb{S}^1, h^L]$ verifies 
	\begin{equation}
		|\sigma(x)|_{h^L_0} = 1, \text{ for } x \in \mathbb{D}.
	\end{equation}
	Note that since $h^L_0 \geq h^L$ over $\mathbb{D}$, we deduce that $P[\mathbb{D}, h^L] = h^L_0$. 
	Hence the contact set $\{ x \in X : P[\mathbb{D}, h^L] = h^L \}$ is given by the unit circle $\mathbb{S}^1 \subset \comp$.
	Since the measure $c_1(L, P[\mathbb{D}, h^L])$ is supported on the contact set, cf. \cite[Exercise 5.10]{GuedjZeriahBook}, by the $\mathbb{S}^1$-symmetry, we obtain that the equilibrium measure associated with $(\mathbb{D}, h^L)$ is given by the Lebesgue measure on $\mathbb{S}^1 \subset \comp$.
	By \cite{BloomLeven1} or \cite{BerBoucNys}, we see that if we denote by $\lambda$ the Lebesgue measure on $\mathbb{D} \subset \comp$, and by $B_k$ the Bergman kernel on $(X, L)$ associated with $(\lambda, h^L)$, then the sequence of probability measures $\frac{1}{k + 1} B_k(z,z) d \lambda(z)$ on $X$ converges weakly, as $k \to \infty$, to the Lebesgue measure on the unit circle $\mathbb{S}^1 \subset \comp$.
	\par 
	In particular, we can find a sequence $a_i, b_i \in ]0, 1[$, $i \in \nat$, verifying $a_i < b_i < a_{i + 1}$, so that if we denote $A_i = \{ z \in \mathbb{D}: a_i < |z| < b_i \}$, $C_i = \{ z \in \mathbb{D}: b_i < |z| < a_{i + 1} \}$, then for certain increasing subsequences $\alpha_i, \beta_i \in \nat$, we have
	\begin{equation}
		\frac{1}{\alpha_i + 1}
		\int_{A_i} B_{\alpha_i}(z,z) d \lambda(z) \geq \frac{2}{3},
		\qquad 
		\frac{1}{\beta_i + 1}
		\int_{C_i} B_{\beta_i}(z,z) d \lambda(z) \geq \frac{2}{3}.
	\end{equation}
	Using the indicator function $1_{A_i} : \comp \to \{0, 1\}$ of $A_i$, we define the function $g : D \to \{0, 1\}$ as 
	\begin{equation}
		g = \sum_{i = 1}^{+\infty} 1_{A_i}.
	\end{equation}
	\par 
	Now, let us construct a linear series for which partial Bergman kernels diverge.
	Consider the manifold $X = \mathbb{P}^1 \sqcup \mathbb{P}^1$, endowed with a line bundle $L$, which restricts to $\mathscr{O}(1)$ on each of the components.
	Let $\pi : X \to \mathbb{P}^1$ be the natural projection.
	Define $W_k \subset H^0(X, L^{\otimes k})$ as
	\begin{equation}\label{eq_pull_bc_prj_space}
		W_k := \pi^* H^0(\mathbb{P}^1, \mathscr{O}(1)^{\otimes k}).
	\end{equation}
	\par 
	Let us consider the following measure $\mu$ on $X$.
	We denote by $\mu_i$, $i = 0, 1$, the restriction of $\mu$ to each of the $\mathbb{P}^1$-components.
	Define $\mu$ so that 
	\begin{equation}
		\mu_1 := \frac{2 + g}{4} \cdot d \lambda, \qquad \mu_2 := \frac{2 - g}{4} \cdot d \lambda,
	\end{equation}
	where $\lambda$ is the Lebesgue measure on the unit disc $\mathbb{D} \subset \comp \subset \mathbb{P}^1$.
	We claim that the measure $\mu$ satisfies the Bernstein-Markov property.
	\par 
	To see this, it is enough to verify the property on each component individually. 
	On each component, $\mu$ is bounded below by the measure $\frac{1}{4} \lambda$, which is known to be Bernstein-Markov, cf. \cite{KlimekBook}. 
	It is also bounded from above by $\lambda$, which is a probability measure.
	Therefore, the Bernstein-Markov property for $\mu$ follows immediately.
	\par 
	We claim that the partial Bergman kernel associated with $W$, $\mu$ and $h^L$ does not converge weakly.
	To verify this, note that since we have $\pi_* \mu = \lambda$, the associated partial Bergman kernel, $B_k[W, \mu, h^L](x, x)$ writes as 
	\begin{equation}\label{eq_part_berg_expl_cal}
		B_k[W, \mu, h^L](x, x)
		=
		B_k(\pi(x), \pi(x)).
	\end{equation}
	Let us pick the (continuous) function $f$ on $X$, which is given by the indicator function on the first $\mathbb{P}^1$-component.
	Then we clearly have
	\begin{multline}
		\int_X f(x) \cdot B_k[W, \mu, h^L](x, x) \cdot d \mu(x)
		=
		\frac{1}{4}
		\int_{\mathbb{D}} (2 + g(x)) \cdot B_k(x, x) \cdot d \lambda(x)
		\\
		=
		\frac{k + 1}{2}
		+
		\frac{1}{4}
		\int_{\mathbb{D}} g(x) \cdot B_k(x, x) \cdot d \lambda(x).
	\end{multline}
	But then from the construction of $\alpha_i$, $\beta_i$, we obtain that for any $i \in \nat$, we have
	\begin{equation}
	\begin{aligned}
		&
		\int_{\mathbb{D}} g(x) \cdot B_{\alpha_i}(x, x) \cdot d \lambda(x)
		\geq
		\int_{A_i} B_{\alpha_i}(x, x) \cdot d \lambda(x)
		\geq
		\frac{2}{3} \cdot (\alpha_i + 1)
		\\
		&
		\int_{\mathbb{D}} g(x) \cdot B_{\beta_i}(x, x) \cdot d \lambda(x)
		\leq
		(\beta_i + 1) - \int_{C_i} B_{\beta_i}(x, x) \cdot d \lambda(x)
		\leq
		\frac{1}{3} \cdot (\beta_i + 1),
	\end{aligned}
	\end{equation}
	which shows that the sequence of measures $\frac{1}{k + 1} B_k[W, \mu, h^L](x, x) d \mu(x)$ does not converge weakly, as $k \to \infty$.
	Note that the Kodaira-Iitaka map associated with this linear series is precisely the projection $\pi$, which is not birational.
	\par 
	\textbf{Example.} Non-differentiability of the relative Monge-Ampère $\kappa$-energy.
	\par
	Consider the manifold $X = \mathbb{P}^1 \times \mathbb{P}^1$, endowed with a line bundle $L := \pi^* \mathscr{O}(1)$, where $\pi : X \to \mathbb{P}^1$ is the projection to the first component.
	Let us consider the complete linear series $W_k = H^0(X, L^{\otimes k})$.
	Note that the map $\pi^* : H^0(\mathbb{P}^1, \mathscr{O}(1)^{\otimes k}) \to H^0(X, L^{\otimes k})$ is an isomorphism.
	\par 
	We fix a metric $h^L$ on $L$, and denote by $h^L_t$ the restriction of $h^L$ to $\mathbb{P}^1 \times \{ t \}$.
	We define the metric $\pi_* h^L$ on $\mathscr{O}(1)$ as $\pi_* h^L := \sup_{t \in \mathbb{P}^1} h^L_t$.
	It is then immediate to see from Section \ref{sect_alt_q_sch} that 
	\begin{equation}\label{eq_env_expl_calc}
		P[W, h^L] = \pi^* P[ \pi_* h^L],
	\end{equation}
	where $P[ \pi_* h^L]$ is defined as in (\ref{eq_sic_env}).
	\par 
	Now, we fix a continuous strictly positive function $f$ on $\mathbb{P}^1$ and construct a function $g$ on $X$ which restricts to $f$ over $\mathbb{P}^1 \times \{0\}$, to $-f$ over $\mathbb{P}^1 \times \{\infty\}$ and such that for any $x, t \in \mathbb{P}^1$, we have $|g(x, t)| \leq |f(x)|$. 
	We fix an arbitrary positive metric $h^{\mathscr{O}(1)}$ on $\mathscr{O}(1)$ and consider the pullback metric $h^L := \pi^* h^{\mathscr{O}(1)}$.
	We claim that the relative Monge-Ampère $\kappa$-energy is not differentiable in the direction of $g$ at $h^L$.
	To see this, remark fist that by (\ref{eq_env_expl_calc}), for any metrics $h^L_0$, $h^L_1$ on $L$, the relative Monge-Ampère $\kappa$-energy is related with the relative Monge-Ampère energy as follows 
	\begin{equation}\label{eq_diff_energy_rel_ex}
		\mathscr{E}_{\kappa}(P[W, h^L_0]) - \mathscr{E}_{\kappa}(P[W, h^L_1])
		=
		\mathscr{E}(P[\pi_* h^L_0]) - \mathscr{E}(P[\pi_* h^L_1]).
	\end{equation}
	Now, from the description before, we conclude that 
	\begin{equation}\label{eq_diff_energy_rel_ex2}
		P[\pi_* h^L \cdot \exp(t g)]
		=
		P[h^{\mathscr{O}(1)} \cdot \exp(|t| f)].
	\end{equation}
	Note that the relative Monge-Ampère energy is differentiable by \cite{BermanBouckBalls}, and we have
	\begin{equation}
		\frac{d}{dt}\Big|_{t = 0}
		\Big(
		\mathscr{E}(P[h^{\mathscr{O}(1)} \cdot \exp(t f)])
		-
		\mathscr{E}(P[h^{\mathscr{O}(1)}])
		\Big)
		=
		\int f \cdot c_1(\mathscr{O}(1), h^{\mathscr{O}(1)}).
	\end{equation}
	From this, (\ref{eq_diff_energy_rel_ex}) and (\ref{eq_diff_energy_rel_ex2}),  we see that the right and left derivatives of the relative Monge-Ampère $\kappa$-energy at $t = 0$ do not coincide.
	\par 
	Calculations analogous to the above show that the relative Monge-Amp\`ere $\kappa$-energy at $h^L_0$ is differentiable only along directions given by pull-backs from the base of $\pi$.
	\par 
	\begin{rem}
		As the reader will easily verify, the second example can be adapted to show that the partial Bergman kernels do not converge even for complete linear series.
	\end{rem}

	\subsection{Partial differentiability of the relative Monge-Ampère $\kappa$-energy}\label{sect_different}
	The main goal of this section is to investigate the partial differentiability of the relative Monge-Ampère $\kappa$-energy and to deduce Theorem \ref{thm_part_bk_get} from this.
	\par 
	To state our partial differentiability result, we fix a complex analytic space $Z$ with a holomorphic map $\rho : X \to Z$ which factorizes, for $m \in \nat$ divisible enough, through the rational maps associated with the linear series as in (\ref{eq_rat_maps_z}).
	\begin{thm}\label{thm_differentiab}
		For any continuous $f : Z \to \real$, the relative Monge-Ampère $\kappa$-energy at a fixed continuous metric is differentiable in the direction $\rho^* f$, and in the notations (\ref{eq_quil_pushed_y}), we have 
		\begin{multline}
			\frac{d}{dt}\Big|_{t = 0}
			\Big(
			\mathscr{E}_{\kappa}(P[W, K, h^L \cdot \exp(- 2 t \rho^* f)])
			-
			\mathscr{E}_{\kappa}(P[W, K, h^L])
			\Big)
			\\
			=
			\int_Z f(x) \cdot d \mu_{\mathrm{eq}}[W, Z, K, h^L](x).
		\end{multline}
	\end{thm}
	\begin{rem}\label{rem_indep_mueq}
		a)
		Since the relative Monge-Ampère $\kappa$-energy is independent from the choice of $\omega$ by Remark \ref{rem_ma_kappa_ma_usual}, the same can be said about $\mu_{\mathrm{eq}}[W, Z, K, h^L]$ by Theorem \ref{thm_differentiab}.
		\par 
		b) In particular, if $W$ is birational, then the relative Monge-Ampère $\kappa$-energy at a fixed continuous metric is differentiable in every continuous direction, as one can take $Z = X$.
	\end{rem}
	Before proceeding with a proof of Theorem \ref{thm_differentiab}, let us explain how it implies Theorem \ref{thm_part_bk_get}.
	For this, we shall use the following lemma, the proof of which the reader can find in \cite[Lemma 6.6]{BermanBouckBalls} or in \cite{YuanEquidistr}.
	\begin{lem}\label{lem_yuan}
		Let $f_k: \real \to \real$, $k \in \nat$, be a sequence of concave functions, and let $g: \real \to \real$ be such that $\liminf_{k \to \infty} f_k \geq g$ and $\lim_{k \to \infty} f_k(0) = g(0)$. 
		If $f_k$ and $g$ are differentiable at $0$, then $\lim_{k \to \infty} f'_k(0) = g'(0)$.
	\end{lem}
	\begin{proof}[Proof of Theorem \ref{thm_part_bk_get}]
		Without loosing the generality, we assume that the measure $\mu$ is a probability measure.
		For any $k \in \nat$, $t \in \real$, let us define the following functions
		\begin{equation}
		\begin{aligned}
			&
			f_k(t) 
			:= 
			\frac{1}{k^{\kappa(W) + 1}} 
				\log \Big(  \frac{\vol(\mathbb{B}_k[W, \mu, h^L \cdot \exp(- 2 t \rho^* f)])}{\vol(\mathbb{B}_k[W, \mu, h^L])}			
			\Big),
			\\
			&
			g(t) 
			:= 
			\mathscr{E}_{\kappa}(P[W, K, h^L \cdot \exp(- 2 t \rho^* f)])
			-
			\mathscr{E}_{\kappa}(P[W, K, h^L]),
		\end{aligned}
		\end{equation}
		where $\mathbb{B}_k[W, \mu, h^L_0]$ denotes the unit ball associated with the norm ${\textrm{Hilb}}_k[W](h^L_0, \mu)$ for a metric $h^L_0$ on $L$.
		As $\mu$ is a probability measure, we obviously have $\mathbb{B}_k[W, K, h^L \cdot \exp(- 2 t \rho^* f)] \subset \mathbb{B}_k[W, \mu, h^L \cdot \exp(- 2 t \rho^* f)]$.
		From this, the Bernstein-Markov assumption on $\mu$ and Theorem \ref{thm_volumes_ratio}, we conclude that $\liminf_{k \to \infty} f_k \geq g$.
		It is also trivially true that $\lim_{k \to \infty} f_k(0) = g(0) = 0$.
		\par 
		By Theorem \ref{thm_differentiab}, $g$ is differentiable at $t = 0$, and we have
		\begin{equation}
			g'(0) 
			=
			\int_Z f(x) \cdot d \mu_{\mathrm{eq}}[W, Z, K, h^L](x).
		\end{equation}
		An easy calculation, cf. \cite[Lemma 4.1]{BermanBouckBalls}, shows that $f_k$ are also differentiable at $t = 0$, and we have
		\begin{equation}
			f'_k(t)
			=
			\frac{1}{k^{\kappa(W)}}
			\int_X
			\rho^* f(x) \cdot 
			B_k[W, \mu, h^L](x, x) \cdot d\mu(x).
		\end{equation}
		It follows also from \cite[Proposition 2.4]{BerBoucNys} that $f_k(t)$ is concave.
		Lemma \ref{lem_yuan} now finishes the proof of Theorem \ref{thm_part_bk_get}.
	\end{proof}
	\par 
	The proof of Theorem \ref{thm_differentiab} relies on a singular version of the differentiability of the relative Monge-Ampère energy due to Berman-Boucksom \cite{BermanBouckBalls} that we now establish.
	We fix a complex manifold $\widehat{Y}$ of dimension $\kappa$ and a big line bundle $\widehat{A}$ over it.
	Consider an upper semicontinuous semimetric $\widehat{h}^A$, on $\widehat{A}$ which is bounded from below over a non-pluripolar subset by a continuous metric.
	By Lemma \ref{lem_env_sing_vers}, $P[\widehat{h}^A]_*$ has a psh potentials.	
	\begin{thm}\label{thm_diff_sing}
		The relative Monge-Ampère energy at $\widehat{h}^A$ is differentiable in continuous directions.
		More precisely, for any continuous $f: \widehat{Y} \to \real$, we have
		\begin{equation}
			\frac{d}{dt}\Big|_{t = 0}
			\Big(
			\mathscr{E}(P[\widehat{h}^A \cdot \exp(- 2t f)])
			-
			\mathscr{E}(P[\widehat{h}^A])
			\Big)
			=
			\int_{\widehat{Y}} f(x) \cdot c_1(\widehat{A}, P[\widehat{h}^A]_*)^{\kappa}.
		\end{equation}
	\end{thm}
	\begin{rem}
		The authors of \cite{BermanBouckBalls} established Theorem \ref{thm_diff_sing} for continuous metrics  $\widehat{h}^A$.
		In our proof, we follow closely the approach of Lu-Nguyen \cite{LuNgueynDifferent}, which simplifies the original proof.
	\end{rem}
	\begin{proof}
		We will be brief and only indicate the necessary changes, see \cite[Theorem 11.11]{GuedjZeriahBook} for details.
		By the concavity of the relative Monge-Ampère energy, we deduce that if $t > 0$, then
		\begin{equation}
			\frac{\mathscr{E}(P[\widehat{h}^A \cdot \exp(- 2t f)]_*) - \mathscr{E}(P[\widehat{h}^A]_*)}{t}
			\leq
			\int_{\widehat{Y}} f(x) \cdot c_1(\widehat{A}, P[\widehat{h}^A]_*)^{\kappa}.
		\end{equation}
		Again, the concavity implies that if $t > 0$, then
		\begin{multline}
			\frac{\mathscr{E}(P[\widehat{h}^A \cdot \exp(- 2t f)]_*) - \mathscr{E}(P[\widehat{h}^A]_*)}{t}
			\\
			\geq
			\int_{\widehat{Y}} \frac{1}{t} \log \Big(\frac{P[\widehat{h}^A]_*}{P[\widehat{h}^A \cdot \exp(- 2t f)]_*} \Big) \cdot c_1(\widehat{A}, P[\widehat{h}^A \cdot \exp(- 2 t f)]_*)^{\kappa}.
		\end{multline}
		However, by the upper semicontinuity of $\widehat{h}^A$ and \cite[Lemma 2.3]{GuedjLuZeriahEnv}, cf. \cite{BedfordTaylor}, the measure $c_1(\widehat{A}, P[\widehat{h}^A \cdot \exp(- 2 t f)]_*)^{\kappa}$ puts no mass  outside of the set $P[\widehat{h}^A \cdot \exp(- 2t f)]_* = \widehat{h}^A \cdot \exp(- 2t f)$.
		By this and the fact that $P[\widehat{h}^A]_* \geq \widehat{h}^A$ outside of a pluripolar subset, we conclude that 
		\begin{equation}
			\frac{\mathscr{E}(P[\widehat{h}^A \cdot \exp(- 2t f)]_*) - \mathscr{E}(P[\widehat{h}^A]_*)}{t}
			\geq
			\int_{\widehat{Y}} f(x) \cdot c_1(\widehat{A}, P[\widehat{h}^A \cdot \exp(- 2 t f)]_*)^{\kappa}.
		\end{equation}
		The proof is now finished by using the continuity of the Monge-Ampère operator with respect to uniform convergence of $P[\widehat{h}^A \cdot \exp(- 2 t f)]_*$ towards $P[\widehat{h}^A]_*$, as $t \to 0$, in exactly the same manner as in \cite[Theorem 11.11]{GuedjZeriahBook}.
	\end{proof}
	\begin{proof}[Proof of Theorem \ref{thm_differentiab}]
		Let us first establish that in the notations as in Lemma \ref{lem_ma_kappa_ma_usual}, for every $\epsilon > 0$, there is $m \in \nat^*$ that is divisible by any prescribed positive integer, so that for any $t \in \real$, we have
		\begin{multline}\label{eq_approx_diff_kappa_relma}
			\Big|
				\big(
				\mathscr{E}_{\kappa}(P[W, K, h^L \cdot \exp(- 2 t \rho^* f)])
				-
				\mathscr{E}_{\kappa}(P[W, K, h^L])
				\big)
				\\
				-
				\frac{1}{m^{\kappa(W) + 1}}
				\big(
				\mathscr{E}(P[\widehat{\varphi}_{m, *} (\widehat{K}, \widehat{h}^L \cdot \exp(- 2 t \rho^* f))]_*)
				-
				\mathscr{E}(P[\widehat{\varphi}_{m, *} (\widehat{K}, \widehat{h}^L]_*)
				\big)
			\Big|
			\leq 
			\epsilon \cdot |t|.
		\end{multline}
		\par 
		For this, note that in the notations of (\ref{eq_thm_volumes_ratio_0}), by Theorem \ref{thm_sym_subalgebra}, Lemma \ref{lem_comp_rest} and (\ref{eq_thm_volumes_ratio_1}), for every $\epsilon > 0$, there are $m \in \nat$, $k_0 \in \nat$, so that for any $t \geq 0$, $k \geq k_0$, we have
		\begin{multline}
			\Big|
			\log \Big(\frac{\vol(\mathbb{B}_{km}[W, K, h^L  \cdot \exp(- 2 t \rho^* f)])}{\vol(\mathbb{B}_{km}[W, K, h^L])} \Big)
			\\
			-
			\log\Big( 
			\frac{\vol(\mathbb{B}_k[Y_m, \widehat{\varphi}_{m, *} (\widehat{K}, \widehat{h}^L \cdot \exp(- 2 t \rho^* f))])}{\vol(\mathbb{B}_k[Y_m, \widehat{\varphi}_{m, *} (\widehat{K}, \widehat{h}^L)])}
			\Big)
			\Big|
			\\
			\leq
			80 \cdot (km)^{\kappa(W)} \cdot \log(km)
			+
			\epsilon \cdot |t| \cdot (km)^{\kappa(W) + 1}.
		\end{multline}
		The estimate (\ref{eq_approx_diff_kappa_relma}) then follows from this by dividing by $(km)^{\kappa(W) + 1}$, taking a limit $k \to \infty$, and using Theorem \ref{thm_volumes_ratio} and Lemma \ref{lem_bb_sing}.
		\par 
		By our assumption (\ref{eq_rat_maps_z}) on $Z$, we immediately have
		\begin{equation}\label{eq_env_birrrr}
			P[\widehat{\varphi}_{m, *} (\widehat{K}, \widehat{h}^L \cdot \exp(- 2 t \rho^* f))]
			=
			P[\widehat{\varphi}_{m, *} (\widehat{K}, \widehat{h}^L) \cdot \exp(- 2 t \psi_m^* f)],
		\end{equation}
		where $\psi_m^* f$ is a continuous function defined on a birational model of $\widehat{Y}_m$, given by the resolution of the indeterminacies of the rational map $\widehat{Y}_m \dashrightarrow Z$.
		Note that despite the fact that $\psi_m^* f$ is defined on a birational model of $\widehat{Y}_m$, the envelope (\ref{eq_env_birrrr}) can be seen as a metric on $\widehat{Y}_m$ by the argument similar to the one after (\ref{eq_hlm_defn}).
		By this, Theorem \ref{thm_diff_sing} and (\ref{eq_approx_diff_kappa_relma}), we deduce that 
		\begin{multline}\label{eq_bound_der_ekap_e_us}
			\limsup_{t \to \infty}
			\frac{\mathscr{E}_{\kappa}(P[W, K, h^L \cdot \exp(- 2 t \rho^* f)])
				-
				\mathscr{E}_{\kappa}(P[W, K, h^L])}{t}
			\\
			\leq
			\epsilon 
			+
			\frac{1}{m^{\kappa(W)}} \cdot
			\int_{\widehat{Y}_m} \psi_m^* f(x) \cdot c_1(\widehat{A}_m, P[\widehat{\varphi}_{m, *} (\widehat{K}, \widehat{h}^L)]_*)^{\kappa(W)},
		\end{multline}		 
		and that a similar bound with an inverse inequality for the $\liminf$ holds.
		Note that since the non-pluripolar product puts no mass on analytic subsets, the right-hand side of (\ref{eq_bound_der_ekap_e_us}) is well-defined.
		By proceeding as in the proof of Lemma \ref{lem_ma_kappa_ma_usual}, we deduce
		\begin{multline}
			\lim_{m \to \infty}
			\frac{1}{m^{\kappa(W)}} \cdot
			\int_{\widehat{Y}_m} \psi_m^* f(x) \cdot c_1(\widehat{A}_m, P[\widehat{\varphi}_{m, *} (\widehat{K}, \widehat{h}^L)]_*)^{\kappa(W)}
			\\
			=
			\int_Z  f(x) \cdot d \mu_{\mathrm{eq}}[W, Z, K, h^L](x).
		\end{multline}
		From this, (\ref{eq_bound_der_ekap_e_us}) and its analogue for the $\liminf$, we deduce the result.
	\end{proof}

	\section{Integral closure of linear series and singularity types}\label{sect_lin_ser_sing_type}
	The main goal of this section is to present an application of Theorems \ref{thm_ki_form} and \ref{thm_vol_form} to the study of the integral closure of linear series -- most notably, to prove Theorem \ref{thm_integ_closuse} -- as well as to investigate linear series constructed from a given singularity type and to prove Theorem \ref{thm_kappa_wphi}. 
	Along the way, we introduce the notion of analytic closure of a linear series and relate it to the integral closure.
	\par 
	\begin{proof}[Proof of Theorem \ref{thm_kappa_wphi}]
		We fix a smooth metric $h^L$ on $L$, let $\alpha = c_1(L, h^L)$, and fix a function $\phi \in \psh(X, \alpha)$. 
		We follow the notations introduced in (\ref{eq_wkphi_defn}).
		\par 
		Recall the following envelope construction of Ross-Witt Nystr{\"o}m \cite{RossWNEnvelop}:
		\begin{equation}
			P[\phi] := 
			\Big(
			\lim_{C \to + \infty}
			\sup\Big\{
				\psi \in \psh(X, \alpha) : \psi \leq \min\{ \phi + C, 0 \}
			\Big\}
			\Big)^*.
		\end{equation}
		As it was established in \cite{RossWNEnvelop}, we have $P[\phi] \in \psh(X, \alpha)$.
		Note also that $P[\phi]$ depends solely on the singularity type of $\phi$, as suggested in the notation.
		\par 
		We define $\phi_k : X \to \real \cup \{ -\infty \}$ so that $FS(\textrm{Ban}_k^{\infty}[W(\phi)](h^L))^{\frac{1}{k}} = h^L \cdot \exp(- 2 \phi_k)$.
		Then immediately from the definitions, we have the following formula
		\begin{equation}\label{eq_phi_k_defn_sup_wkphi}
			 \phi_k(x)
			 =
			 \Big(
			 \sup_{s \in W_k(\phi)} \Big\{
			 	\frac{1}{k} \log |s(x)|_{(h^L)^k} : \sup_{y \in X} |s(y)|_{(h^L)^k} \leq 1
			\Big\}
			\Big)^*.
		\end{equation}
		By the definitions of $W_k(\phi)$, $P[\phi]$ and (\ref{eq_poinc_lll}), we deduce that 
		\begin{equation}
			\phi_k(x) 
			\leq
			P[\phi].
		\end{equation}
		This implies in particular that for $P[\phi](h^L) := h^L \cdot \exp(- 2 P[\phi])$, we have
		\begin{equation}\label{eq_pphi_comp}
			P[W(\phi), h^L] \geq P[\phi](h^L).
		\end{equation}
		Immediately from Theorem \ref{thm_ki_form}, (\ref{eq_nd_defn}) and (\ref{eq_pphi_comp}), we deduce that for $T[\phi] := c_1(L, P[\phi](h^L))$,
		\begin{equation}\label{eq_kappa_nd_thi}
			\kappa(W(\phi)) \leq {\rm{nd}}(T[\phi]).
		\end{equation}
		However, from Theorem \ref{thm_wn_monot}, it is immediate to see, cf. \cite[Remark 3.4]{DarvDiNezLuSingType}, that for any K\"ahler form $\omega$ on $X$ and a $(1, 1)$-current $T := \alpha + dd^c \phi$, we have
		\begin{equation}\label{eq_mass_env_pres}
			\int T[\phi]^i \wedge \omega^{n - i}
			=
			\int T^i \wedge \omega^{n - i}.
		\end{equation}
		In particular, this shows that ${\rm{nd}}(T[\phi]) = {\rm{nd}}(T)$.
		From this and (\ref{eq_kappa_nd_thi}), we deduce the bound $\kappa(W(\phi)) \leq {\rm{nd}}(T)$.
		The proof of the second bound follows immediately from Theorems \ref{thm_vol_form}, \ref{thm_wn_monot}, (\ref{eq_pphi_comp}) and (\ref{eq_mass_env_pres}).
	\end{proof}
	\par 
	Let us now introduce the \textit{analytic closure}, $\widetilde{W}$, of a linear series $W \subset R(X, L)$.
	For this, we proceed as follows.
	We fix a smooth metric $h^L$ on $L$, and write $P[W, h^L] = h^L \cdot \exp(- 2 \phi(W))$. 
	We denote $\alpha := c_1(L, h^L)$; then $\phi(W) \in \psh(X, \alpha)$ by (\ref{eq_poinc_lll}).
	Relying on the construction of the linear series associated with the potential $\phi(W)$, introduced before Theorem \ref{thm_kappa_wphi}, we define
	\begin{equation}
		\widetilde{W} := W(\phi(W)).
	\end{equation}
	Note that since the singularity type does not depend on the choice of $h^L$, the linear series $\widetilde{W}$ is independent of the choice of it as well.
	In the following result, which clearly generalizes Theorem \ref{thm_integ_closuse}, the relative degrees are defined analogously to the definition preceding Theorem \ref{thm_integ_closuse}.
	\begin{thm}\label{thm_anal_clos}
		The analytic closure contains the integral closure, i.e. $\overline{W} \subset \widetilde{W}$.
		Moreover, we have $\kappa(\widetilde{W}) = \kappa(W)$.
		Finally, for any linear series $W_0 \subset R(X, L)$ verifying $W \subset W_0 \subset \widetilde{W}$, we have ${\rm{vol}}_{\kappa}(W_0) = {\rm{deg}}(W_0 : W) \cdot {\rm{vol}}_{\kappa}(W)$.
	\end{thm}
	\begin{rem}
		Note that when $W$ is birational, we immediately have ${\rm{deg}}(W_0 : W) = 1$.
	\end{rem}
	\begin{proof}
		Let us first establish that $W \subset \widetilde{W}$.
		For this, note that $\phi(W)$ admits the following description
		\begin{equation}
			\phi(W)(x)
			=
			\Big(
				\lim_{k \to \infty} \sup_{s \in W_k} \big\{
					\frac{1}{k} \log|s(x)|_{(h^L)^k} : 
					\sup_{y \in X} |s(y)|_{(h^L)^k} \leq 1
				\big\}
			\Big)^*.
		\end{equation}
		In particular, we see that for any $s \in W_k$, there is $C > 0$, so that for any $x \in X$, we have
		\begin{equation}
			\frac{1}{k} \log|s(x)|_{(h^L)^k}
			\leq 
			\phi(W)(x) + C,
		\end{equation}
		implying the inclusion $W \subset \widetilde{W}$.
		\par 
		Let us now establish that $\overline{W} \subset \widetilde{W}$.
		For this, we fix $s \in \overline{W}_k$.
		Then by the definition of $\overline{W}$ there are $r \in \nat^*$, $a_i \in W_{ik}$, $i = 0, \ldots, r - 1$, so that 
		\begin{equation}
			s^r = \sum_{i = 0}^{r - 1} a_{r - i} \cdot s^i.
		\end{equation}
		We consider the subsets $E_i$, $i = 0, \ldots, r - 1$, defined by the condition that $x \in E_i$ if the maximum of the values $|a_{k - j}(x) \cdot s^j(x)|_{(h^L)^{rk}}$ among all $j = 0, \ldots, r - 1$ is achieved at $j = i$.
		Clearly $E_i$, $i = 0, \ldots, r - 1$, cover $X$, and for any $x \in E_i$, we have
		\begin{equation}
			|s(x)^r|_{(h^L)^{rk}} \leq r \cdot |a_{r - i}(x) \cdot s^i(x)|_{(h^L)^{rk}}.
		\end{equation}
		Since $a_i \in W_{ik}$, and $W_{ik} \subset \widetilde{W}_{ik}$, we deduce that there are $C_i > 0$, so that over $E_i$, we have
		\begin{equation}
			\frac{1}{k} \log |s(x)|_{(h^L)^k}
			\leq
			\phi(W)(x) + C_i.
		\end{equation}
		By taking $C := \max_{i = 0}^{r - 1} C_i$, we deduce that the above bound holds for any $x \in X$ with $C$ in place of $C_i$.
		This clearly shows that $s \in \widetilde{W}$, and it implies $\overline{W} \subset \widetilde{W}$.
		\par 
		Note also that we clearly have $\kappa(W) \leq \kappa(\widetilde{W})$.
		On another hand, from Theorem \ref{thm_kappa_wphi}, we deduce that $\kappa(\widetilde{W}) \leq {\rm{nd}}(c_1(L, P[W, h^L]))$.
		But Theorem \ref{thm_ki_form} shows that  ${\rm{nd}}(c_1(L, P[W, h^L])) = \kappa(W)$, finishing the proof of the identity $\kappa(\widetilde{W}) = \kappa(W)$.
		\par 
		To establish the last part, note that by (\ref{eq_fs_monoton2}), we have $P[W, h^L] \geq P[\widetilde{W}, h^L]$.
		From this, (\ref{eq_pphi_comp}) and (\ref{eq_mass_env_pres}), we see that the integral on the right hand side of Theorem \ref{thm_vol_form} coincides for the linear series $W, \widetilde{W}$.
		Since the envelope construction is obviously monotonic, we deduce by Theorem \ref{thm_wn_monot} that the integral term from Theorem \ref{thm_vol_form} coincides for all the linear series $W, W_0, \widetilde{W}$.
		We denote by $Y_m$, $Y_m^0$, $\widetilde{Y}_m$ the closures of the images of $X$ under the rational maps associated with the linear series $W, W_0, \widetilde{W}$ respectively.
		The diagram (\ref{eq_lin_ser_emb}) shows that we have the rational maps $\widetilde{Y}_m \dashrightarrow Y_m^0 \dashrightarrow Y_m$.
		The first part of Theorem \ref{thm_anal_clos} as well as (\ref{eq_dim_ki_yk}), imply that $\dim \widetilde{Y}_m = \dim Y_m^0 = \dim Y_m$, and so the above rational maps are generically finite.
		Immediately from the definitions, we see that the ratio between the denominator contributions in the formulas of Theorem \ref{thm_vol_form} corresponding to $W, W_0, \widetilde{W}$, can be measured as the generic degrees of these maps.
	\end{proof}

\bibliography{bibliography}

		\bibliographystyle{abbrv}

\Addresses

\end{document}